\documentclass[11pt]{amsart}

\usepackage{amsaddr}
\usepackage{amscd}
\usepackage{amsfonts}
\usepackage{amsmath}
\usepackage{amssymb}
\usepackage{amsthm}
\usepackage{bbm} 
\usepackage{cite}
\usepackage{color}
\usepackage{cases}
\usepackage{latexsym}
\usepackage{mathtools}
\usepackage{microtype}
\usepackage[active]{srcltx}
\usepackage[UKenglish]{babel}
\usepackage{url}

\usepackage{hyperref}
\usepackage{enumitem}
\usepackage{cleveref}

\usepackage[charter]{mathdesign}

\usepackage[mathscr]{euscript}

\DeclareMathAlphabet{\mathcalalt}{OMS}{cmsy}{m}{n}

\newcommand{\poalgfont}{\mathcalalt}
\newcommand{\bbfont}{\mathbbm}
\newcommand{\lebfont}{\mathcal}
\newcommand{\opstrufont}{\mathcalalt}

\allowdisplaybreaks
\numberwithin{equation}{section}

\newcommand{\upc}{{\mathrm{c}}}
\newcommand{\upd}{{\mathrm{d}}}
\newcommand{\upC}{{\mathrm{C}}}

\newcommand{\NN}{{\bbfont N}}
\newcommand{\RR}{{\bbfont R}}

\newcommand{\ulp}{{\textup{(}}}
\newcommand{\urp}{{\textup{)}}}
\newcommand{\uppars}[1]{\ulp #1\urp}

\newcommand{\abs}[1]{{\lvert #1 \rvert}}
\newcommand{\norm}[1]{{\lVert #1 \rVert}}
\newcommand{\lrabs}[1]{{\left\lvert #1 \right\rvert}}

\newcommand{\mc}{mono\-tone com\-plete}
\newcommand{\smc}{$\sigma$-mono\-tone com\-plete}
\newcommand{\Dc}{De\-de\-kind com\-plete}
\newcommand{\sDc}{$\sigma$-De\-de\-kind com\-plete}

\newcommand{\SOTlim}{\mathrm{SOT-}\lim}

\newcommand{\comp}{{\upc}}

\newcommand{\indicator}[1]{\chi_{#1}}

\newcommand{\pos}[1]{{#1^+}}
\newcommand{\negt}[1]{{#1^-}}

\newcommand{\f}[1]{(#1)}
\newcommand{\inp}[1]{\langle #1 \rangle}
\newcommand{\lrinp}[1]{\left\langle #1\right\rangle}

\newcommand{\seq}[1]{\{{#1}_n\}_{n=1}^{\infty}}
\newcommand{\net}[1]{\{{#1}_\lambda\}_{\lambda\in \Lambda}}

\newcommand{\Ell}{{\mathrm{L}}}

\newcommand{\largest}{\infty}

\newcommand{\sa}{{\mathrm{sa}}}

\newcommand{\pset}{X}
\newcommand{\pt}{x}
\newcommand{\ts}{X}

\newcommand{\os}{E}
\newcommand{\ostwo}{F}

\newcommand{\posos}{\pos{\os}}
\newcommand{\osext}{\overline{\os}}
\newcommand{\pososext}{\overline{\posos}}

\newcommand{\posR}{\pos{\RR}}
\newcommand{\Rext}{\overline{\RR}}
\newcommand{\posRext}{\overline{\pos{\RR}}}

\newcommand{\hilbert}{H}
\newcommand{\jbw}{\poalgfont{M}}

\newcommand{\posmap}{\pi}
\newcommand{\posmapT}{T}

\newcommand{\odual}[1]{{#1^{\thicksim}}}
\newcommand{\ndual}[1]{{#1^{\ast}}}
\newcommand{\ocdual}[1]{{#1_{\mathrm {oc}}^{\thicksim}}}
\newcommand{\socdual}[1]{{#1_{\sigma\mathrm {oc}}^{\thicksim}}}

\newcommand{\odualos}{\odual{\os}}
\newcommand{\ndualos}{\ndual{\os}}
\newcommand{\ocdualos}{\ocdual{\os}}
\newcommand{\socdualos}{\socdual{\os}}

\newcommand{\bounded}{{\opstrufont B}}
\newcommand{\linear}{{\opstrufont L}}

\newcommand{\boundedh}{\bounded (\hilbert)}

\newcommand{\obounded}{\linear_{\mathrm{ob}}}
\newcommand{\ocontinuous}{\linear_{\mathrm{oc}}}
\newcommand{\socontinuous}{\linear_{\sigma\mathrm{oc}}}
\newcommand{\linearop}[1]{\linear(#1)}
\newcommand{\oboundedop}[1]{\obounded(#1)}
\newcommand{\regularop}[1]{\linear_{\mathrm r}(#1)}
\newcommand{\ocontop}[1]{\ocontinuous(#1)}
\newcommand{\socontop}[1]{\socontinuous(#1)}

\newcommand{\alg}{\Omega}

\newcommand{\M}{\alg}

\newcommand{\mss}{\Delta}
\newcommand{\msstwo}{\Gamma}
\newcommand{\mssthree}{\Theta}

\newcommand{\ms}{(\pset,\alg)}
\newcommand{\msm}{(\pset,\alg,\npm,\os)}

\newcommand{\npm}{\mu}
\newcommand{\pom}{\npm^\ast}

\newcommand{\opint}[1]{I_{#1}}
\newcommand{\opintm}{\opint{\npm}}

\newcommand{\di}[1]{\,\upd #1}
\newcommand{\orderintegral}[3]{{\int_{#1}^{\mathrm{o}}\! {#2}\di {#3}}}
\newcommand{\ointm}[1]{\orderintegral{\pset}{#1}{\npm}}

\newcommand{\elemfun}{{{\lebfont E}(\pset,\alg;\posR)}}
\newcommand{\posextmeasfun}{{{\lebfont M}(\pset,\alg;\posRext)}}
\newcommand{\measfun}{{{\lebfont M} (\pset,\alg;\RR)}}
\newcommand{\posmeasfun}{{{\lebfont M}(\pset,\alg;\posR)}}
\newcommand{\integrablefun}{{{\lebfont L}^1(\pset,\alg,\npm;\RR)}}
\newcommand{\posintegrablefun}{{{\lebfont L}^1(\pset,\alg,\npm;\posR)}}
\newcommand{\aezerofun}{{{\lebfont N}(\pset,\alg,\npm;\RR)}}

\newcommand{\ellone}{{{\mathrm L}^1(\pset,\alg,\npm;\RR)}}

\newcommand{\cont}[1]{{\upC}(#1)}
\newcommand{\conto}[1]{\upC_0(#1)}
\newcommand{\contc}[1]{\upC_{\upc}(#1)}
\newcommand{\contts}{\cont{\ts}}
\newcommand{\contots}{\conto{\ts}}
\newcommand{\contcts}{\contc{\ts}}

\theoremstyle{plain}

\newtheorem{theorem0}{Theorem}[section]
\newtheorem{theorem}[theorem0]{Theorem}
\newtheorem{proposition}[theorem0]{Proposition}
\newtheorem{lemma}[theorem0]{Lemma}
\newtheorem{corollary}[theorem0]{Corollary}

\newtheorem*{theorem*}{Theorem}
\newtheorem*{proposition*}{Proposition}
\newtheorem*{lemma*}{Lemma}
\newtheorem*{corollary*}{Corollary}

\theoremstyle{definition}

\newtheorem{definition}[theorem0]{Definition}

\newtheorem{remark}[theorem0]{Remark}

\newtheorem*{definition*}{Definition}
\newtheorem*{example*}{Example}
\newtheorem*{remark*}{Remark}

\setlist[enumerate,1]{label=\textup{(\arabic*)},ref=\textup{(\arabic*)}}
\setlist[enumerate,2]{label=\textup{(\alph*)},ref=\textup{(\alph*)}}
\setlist[enumerate,3]{label=\textup{(\roman*)},ref=\textup{(\roman*)}}
\setlist[enumerate,4]{label=\textup{(\Alph*)},ref=\textup{(\Alph*)}}

\crefname{theorem}{Theorem}{Theorems}
\crefname{proposition}{Proposition}{Propositions}
\crefname{lemma}{Lemma}{Lemmas}
\crefname{corollary}{Corollary}{Corollaries}
\crefname{definition}{Definition}{Definitions}
\crefname{example}{Example}{Examples}
\crefname{remark}{Remark}{Remarks}
\crefname{equation}{equation}{equations}

\crefname{section}{Section}{Sections}
\crefname{subsection}{Section}{Sections}
\crefname{subsubsection}{Section}{Sections}

\begin{document}


\title[Order integrals]{Order integrals}

\author{Marcel de Jeu}
\address[Marcel de Jeu]{Mathematical Institute, Leiden University, P.O.\ Box 9512, 2300 RA Leiden, The Netherlands\\
	and\\
	Department of Mathematics and Applied Mathematics, University of Pretoria, Corner of Lynnwood Road and Roper Street, Hatfield 0083, Pretoria,
	South Africa}
\email[Marcel de Jeu]{mdejeu@math.leidenuniv.nl}

\author{Xingni Jiang}
\address[Xingni Jiang]{College of Mathematics, Sichuan University, No.\ 24, South Section, First Ring Road, Chengdu, P.R.\ China}
\email[Xingni Jiang]{x.jiang@scu.edu.cn}


\begin{abstract}
	We define an integral of real-valued functions with respect to a measure that takes its values in the extended positive cone of a partially ordered vector space $E$. The monotone convergence theorem, Fatou's lemma, and the dominated convergence theorem are established; the analogues of the classical ${\mathcal L}^1$- and ${\mathrm L}^1$-spaces are investigated. The results extend earlier work by Wright and specialise to those for the Lebesgue integral when $E$ equals the real numbers.\\
	The hypothesis on $E$ that is needed for the definition of the integral and for the monotone convergence theorem to hold ($\sigma$-monotone completeness) is a rather mild one. It is satisfied, for example, by the space of regular operators between a directed partially ordered vector space and a $\sigma$-monotone complete partially ordered vector space, and by every JBW-algebra. Fatou's lemma and the dominated convergence theorem hold for every $\sigma$-Dedekind complete space.\\
	When $E$ consists of the regular operators on a Banach lattice with an order continuous norm, or when it consists of the self-adjoint elements of a strongly closed complex linear subspace of the bounded operators on a complex Hilbert space, then the finite measures as in the current paper are precisely the strongly $\sigma$-additive positive operator-valued measures. When $E$ is a partially ordered Banach space with a closed positive cone, then every positive vector measure is a measure in our sense, but not conversely. Even when a measure falls into both categories, the domain of the integral as defined in this paper can properly contain that of any reasonably defined integral with respect to the vector measure using Banach space methods.
\end{abstract}

\maketitle

\section{Introduction and overview}\label{1_sec:introduction}

\noindent Let $\ts$ be a locally compact Hausdorff space. The Riesz representation theorem states that, for a positive linear functional $\posmap:\contcts\to\RR$, there exists a Borel measure $\mu$ on $\ts$ such that
\[
\posmap(f)=\int_\ts\! f\di{\npm}
\]
for $f\in\contcts$. The measure $\npm$ is uniquely determined when certain regularity properties of it are supposed. It is bounded if and only if $\posmap$ extends to a  positive linear functional on $\contcts$, in which case the above equation holds for all $f\in\contots$.  In \cite{de_jeu_jiang:2021b}, we shall establish similar representation theorems for positive linear operators  $\posmap:\contcts\to\os$ and $\posmap:\contots\to\os$, where $\os$ is a (suitable) partially ordered vector space.\footnote{In the course of the present paper and its sequels \cite{de_jeu_jiang:2021b,de_jeu_jiang:2021c,de_jeu_jiang:2021d} we shall encounter maps with $\contcts$ or $\contots$ as domains that are sometimes positive linear operators, sometimes vector lattice homomorphisms, and sometimes positive algebra homomorphisms. For each of these contexts, a canonical symbol for such maps could be chosen. However, since our results for these contexts are related, we have chosen to use the same symbol throughout, thus keeping the notation as uniform as possible.} The class of spaces for which these theorems hold is fairly diverse. This is the case, for example, when $\os$ is a Banach lattice with an order continuous norm; when $\os$ consists of the regular operators on a KB-space; when $\os$ is the space of all self-adjoint elements of a strongly closed complex linear subspace of $\boundedh$, where $\hilbert$ is a complex Hilbert space; and when $\os$ is a JBW-algebra.\footnote{
	We shall use \cite{alfsen_shultz_STATE_SPACES_OF_OPERATOR_ALGEBRAS:2001,alfsen_shultz_GEOMETRY_OF_STATE_SPACES_OF_OPERATOR_ALGEBRAS:2003} references for JBW-algebras. In these books, a JBW-algebra is supposed to have an identity element; see \cite[Definitions~1.5 and~2.2]{alfsen_shultz_GEOMETRY_OF_STATE_SPACES_OF_OPERATOR_ALGEBRAS:2003}. In other sources, this need not be the case. However, as \cite[Lemma~4.1.7]{hanche-olsen_stormer_JORDAN_OPERATOR_ALGEBRAS:1984} shows, the existence of an identity element is, in fact, automatic.}
	 In \cite{de_jeu_jiang:2021c}, we shall consider positive algebra homomorphisms $\posmap$ from $\contcts$ or $\contots$ into (suitable) partially ordered algebras. The representing measures are then spectral measures that take values in the algebra. This existence theorem for abstract spectral measures immediately implies the classical one for representations of (the complexification of) $\contots$ on complex Hilbert spaces, as well as the one for positive representations of $\contots$ on KB-spaces in \cite{de_jeu_ruoff:2016}. In \cite{de_jeu_jiang:2021d}, we shall be concerned with representation theorems for vector lattices (resp.\ Banach lattices) of regular linear operators from $\contcts$ and $\contots$ into \Dc\ vector lattices (resp.\ Banach lattices with order continuous norms) in the spirit of \cite[Theorem~38.7]{aliprantis_burkinshaw_PRINCIPLES_OF_REAL_ANALYSIS_THIRD_EDITION:1998}. The relation with existing representation theorems for positive linear operators will be discussed in \cite{de_jeu_jiang:2021b}. There appears to be no previous work in the vein of \cite{de_jeu_jiang:2021c} or \cite{de_jeu_jiang:2021d}.

The representation theorems in \cite{de_jeu_jiang:2021b,de_jeu_jiang:2021c,de_jeu_jiang:2021d} are all of the following form. For a positive operator $\posmap$ from $\contcts$, say, into a partially ordered vector space $\os$, there exists a Borel measure $\npm$ on $\ts$ such that
\[
\posmap(f)=\ointm{f}
\]
for $f\in\contcts$. Here $\npm$ assigns an element of the extended positive cone of $\os$ to each Borel subset of $\ts$. The integral is, as we have called it in the current paper, the order integral\textemdash hence the superscript\textemdash of $f$ with respect to $\npm$. The goal of the present paper is to develop the theory of this order integral to the extent where also the three basic convergence theorems have been made available. It provides the very language for the sequels \cite{de_jeu_jiang:2021b,de_jeu_jiang:2021c,de_jeu_jiang:2021d} to the present paper, but its results\textemdash including the convergence theorems\textemdash may also be of use elsewhere. In fact, as we shall argue in \cref{1_sec:comparison_with_banach_space_case}, when $\os$ is a partially ordered Banach space, then the order integral provides a tool to work with that is better than the integral with respect to positive vector measures.

\medskip

This paper is organised as follows.

In \cref{1_sec:partially_ordered_vector_spaces}, we collect the necessary prerequisites about partially ordered vector spaces. It is explained how a point at infinity can be added to a partially order vector space $\os$, to accommodate the fact that\textemdash as is already obvious from Lebesgue measure on the real line\textemdash representing measures need not be finite. The section also contains a stockpile of technical tools that are helpful when working with the ordering in $\os$ and the extended space $\osext$ in the present paper and its sequels.

\cref{1_sec:normal_monotone_complete_spaces} provides a number of examples of spaces where the representation theorems in \cite{de_jeu_jiang:2021b} are valid.
This material is not yet needed in the current paper. We have, nevertheless, still included it here, to show that there are natural spaces, including spaces of operators, to which the theory of the order integral in this paper applies.

In \cref{1_sec:positive_osext-valued_measure}, measures with values in the extended positive cone $\pososext$ of a partially ordered vector space $\os$ are introduced. The basic (convergence) properties are established and the Borel--Cantelli lemma is proved. For two important examples of spaces of operators, it is shown that the finite measures in our sense are precisely the strongly $\sigma$-additive positive operator-valued measures. In this section, it is still possible to work with algebras rather than $\sigma$-algebras of subsets.

\cref{1_sec:outer_measures} covers outer measures in the context of partially ordered vector spaces. This material is needed in the proof of one of the Riesz representation theorems in \cite{de_jeu_jiang:2021b}, but does not reappear in the present paper.

\cref{1_sec:integration_with_respect_to_a_positive_os-valued_measure} starts with the definition of the order integral for measurable functions with values in the extended positive real numbers; the measure is now supposed to be defined on a $\sigma$-algebra. After that, the monotone convergence theorem, Fatou's lemma and the dominated convergence theorem are established. The section concludes with vector lattice properties of the general ${\lebfont L}^1$- and $\Ell^1$-spaces.

In the final \cref{1_sec:comparison_with_banach_space_case}, we consider the situation when the partially ordered vector space happens to be a Banach space with a closed positive cone. In this case, one can also speak of positive vector measures and ask for the relation between such measures and the positive measures in \cref{1_sec:positive_osext-valued_measure}, and also for the relation between their integrals. It will become clear that every positive vector measure is a positive measure as in \cref{1_sec:positive_osext-valued_measure}, but not conversely. It can occur that a measure falls in both categories, while the order integral properly extends any integral that one may reasonably define using Banach space methods.

\medskip

It appears that Wright was the first to realise how the $\sigma$-additivity of a measure with values in the extended positive real numbers can be generalised to measures with values in the extended positive cone of a  partially ordered vector space $\os$, and that a theory of integration can be built on this.  This was first done in \cite{wright:1969} when $\os$ is a Stone algebra, i.e., a Banach lattice algebra of the form $\contts$ for an extremally disconnected compact Hausdorff space $\ts$. It is also mentioned there that this can be done equally well if $\os$ is a \sDc\ Banach lattice. Topology is no longer present in \cite[p.~193]{wright:1971a}, where it is noted that an order integral can be defined if $\os$ is a \sDc\ vector lattice, and that analogues of the Lebesgue convergence theorems can be obtained. Details are, however, not included. In \cite[p.~678]{wright:1972}, the measures are defined in the most general context\textemdash that of a \smc\footnote{See \cref{1_def:order_completeness}.} partially ordered vector space\textemdash where this definition is meaningful. \cref{1_def:positive_pososext_valued_measure} in the current paper is taken from that source. An order integral is defined in \cite{wright:1972} and a monotone convergence theorem is established. Fatou's lemma and the dominated convergence theorem are alluded to as more complicated results to be worked on later; the outcome appears not to have been published.

Apart from the facts that we extend the theory well beyond that in \cite{wright:1969,wright:1971a,wright:1972} by including results such as the Borel--Cantelli lemma, Fatou's lemma, the dominated convergence theorem, as well as material on outer measures and vector lattice properties of ${\lebfont L}^1$- and $\Ell^1$-spaces, there are also two important differences between the approach in \cite{wright:1969,wright:1971a,wright:1972} and that in our work. Firstly, we define the order integral for measurable functions that take values in the \emph{extended} positive real numbers. The presence of an extra `infinity' for functions besides the one for the measure is technically a little more complicated than when working with finite-valued functions. It is, however, desirable, to allow this so that the sharpest versions of the monotone convergence theorem and Fatou's lemma can be formulated and proved. We shall benefit from this when studying ups and downs in \cite{de_jeu_jiang:2021c}. Secondly, we believe that our approach to the definition of the order integral of an (extended) positive measurable function is more natural. In \cite{wright:1969,wright:1972}, the integral of a measurable function $f:\pset\to\posR$ is defined to be infinity when there exists a $c>0$ such that $\{x\in \pset : f(x)>c\}$ has infinite measure. When this is not the case, then $f$ can be approximated pointwise from below by (finite-valued) elementary functions that have supports with finite measure; the integral of $f$ is then defined using such approximants. In our approach, the integral of a measurable function with values in the extended positive real numbers is always defined using (finite-valued) elementary functions. Such an elementary function is allowed to have a support with infinite measure, in which case its integral is automatically infinity as a consequence of the action of the positive real numbers on the extended positive cone $\pososext$ of $\os$. Provided one argues carefully with the operations on and the ordering of this extended cone, one thus obtains a completely natural integration theory for measurable functions with values in the extended positive real numbers. This is then used as a the starting point for finite-valued measurable functions.

\medskip

The order integral as defined in the present paper has the usual Lebesgue integral as a special case. It will become clear that, in the end, it is possible to choose arguments for the real case that, after being reformulated and adapted appropriately, yield valid proofs for the general case. Still, there are a few caveats when glancing over a proof for the real case and concluding, perhaps all too quickly, that the result holds more generally. Consider, for example, the fact that a measurable function that is almost everywhere equal to zero has zero integral. For positive functions, this follows from the definitions. For general functions, this then `obviously' follows from the inequality  $\abs{\int f\di{\npm}}\leq\int\abs{ f}\di{\npm}$. This inequality is, however, meaningless for general partially ordered vector spaces, which implies that the proof of \cite[Proposition~2.23.b]{folland_REAL_ANALYSIS_SECOND_EDITION:1999} cannot be used in the general case. Even though it is not difficult to remedy this, it still shows that it is easy to make mistakes when thinking that results from the real case are `clearly' also true in general. The real numbers form a topological algebra, a complete metric space, and their partial ordering is a linear ordering. Arguments that rely on these properties\textemdash which are entrenched in our way of thinking\textemdash are to be circumvented for the general theory. One has to (be enabled to) convince oneself\textemdash with the formal technical tools in \cref{1_sec:partially_ordered_vector_spaces} at hand\textemdash that it is actually possible to do this. It is for this reason that we have given proofs of all results.\footnote{\cref{1_res:integrals_of_elementary_functions} and a few statements surrounding it form an exception.} Some of them\textemdash and in particular arguments of an algebraic nature that work for any commutative monoid\textemdash are identical to those for the real case. Many of them, however, have to be adapted to some extent for the general context. We did not want to necessitate the reader to keep moving back and forth between other sources and the paper and have, therefore, kept the latter self-contained.

\medskip

\begin{remark*}\label{1_rem:reverse_roles}
In the current paper, the function is real-valued and the measure takes its values in (the extension of) a partially ordered vector space. These roles can be reversed. For this set-up, the reader is referred to \cite{groenewegen_THESIS:1983,jeurnink_THESIS:1983,kaplin_fijavz_2020:generation_of_relatively_uniformly_continuous_semigroups_on_vector_lattices,van_rooij_van_zuijlen:2016,van_rooij_van_zuijlen:2017}.
\end{remark*}

\section{Partially ordered vector spaces}
\label{1_sec:partially_ordered_vector_spaces}

\noindent In this section, we establish some terminology and notation for partially ordered vector spaces, and collect a number of technical facts. We introduce various types of order completeness and relate these to the extended space that is obtained by adjoining a point at infinity.

\medskip

Unless otherwise indicated, all vector spaces we shall consider are over the real numbers. Operators between two vector spaces are always supposed to be linear, as are functionals. All vector lattices are supposed to be Archimedean.

If $\os$ is a partially ordered vector space, then $\pos{\os}$ denotes its positive cone.  We do not require that $\pos{\os}$ be generating, i.e.,  it need not be the case that $\os=\pos{\os}-\pos{\os}$. Equivalently, we do not require that $\os$ be directed. We do require, however, that $\pos{\os}$ be proper, i.e., that $\pos{\os}\cap\left(-\pos{\os}\right)=\{0\}$.

Subsets of a partially ordered vector space $\os$ of the form $\{x\in\os: a\leq x\leq b\}$ for $a,b\in \os$ such that $a\leq b$, are called \emph{order intervals} in $\os$; they are denoted by $[a,b]$. A subset of $\os$ is called \emph{order bounded} if it is contained in an order interval.

\medskip

 Order completeness properties of partially ordered vector spaces are at the heart of the current paper and the sequels \cite{de_jeu_jiang:2021b,de_jeu_jiang:2021c,de_jeu_jiang:2021d}.
 We list them in  the following definition, which also contains some self-evident notation that we shall use. Index sets for nets are supposed to be partially ordered, and not just pre-ordered.

\begin{definition}\label{1_def:order_completeness} A partially ordered vector space $\os$ is called
\begin{enumerate}
\item\label{1_part:order_completeness_1}
\emph{\smc} if every increasing sequence $\seq{a}$ in $\os$ that is bounded from above has a supremum $\bigvee_{n=1}^{\infty}a_n$ in $\os$;
\item\label{1_part:order_completeness_2}
\emph{\mc} if every increasing net $\net{a}$ in $\os$ that is bounded from above has a supremum $\bigvee_{\lambda\in\Lambda}a_\lambda$ in $\os$;
\item\label{1_part:order_completeness_3}
\emph{\sDc} if every non-empty at most countably infinite subset $S$ of $\os$ that is bounded from above has a supremum $\bigvee\{x : x\in S \}$ in $\os$;
\item\label{1_part:order_completeness_4}
\emph{\Dc} if every non-empty subset $S$ of $\os$ that is bounded from above has a supremum $\bigvee\{x : x\in S\}$ in $\os$.
\end{enumerate}
\end{definition}

Equivalently, one can define these properties by requiring the existence of infima when replacing `increasing' with `decreasing' and `bounded from above' with `boun\-ded from below'. Still equivalently, one can define these properties under the supposition that the sequence, net, or subset is contained in $\posos$.

There are evident logical implications between these four properties. For vector lattices, \Dc ness (resp.\ \sDc ness) and \mc ness (resp.\ \smc ness) are equivalent. If $\os$ is directed and if $\os$ is \sDc, then, for every $x_1,x_2\in\os$, the subset $\{x_1,x_2\}$ is bounded from above, so that it has a supremum. Hence $\os$ is then a vector lattice.

\medskip

 A partially ordered vector space $\os$ is \emph{Archimedean} if $\bigwedge\{\varepsilon x : \varepsilon >0\}=0$ for all $x\in\pos{\os}$. One can equivalently require that $\bigwedge \{r_n x : n\in\NN\}=0$ for every $x\in\pos{\os}$ and every (or just one) sequence $\{r_n\}_{n=1}^{\infty}\subseteq\posR\setminus\{0\}$ such that $r_n\downarrow 0$. Still equivalently, one can require that, whenever $y\in\pos{\os}$ and $x$ are such that $nx\leq y$ for all $n\in\NN$ (or such that $rx\leq y$ for all $r\in\posR$), it follows that $x\leq 0$. As was observed in \cite[Lemma~1.1]{wright:1972}, every \smc\  partially vector space (and then also every \mc, \sDc, or \Dc\  partially order vector space) is Archimedean. Indeed, if $x\geq 0$, then $\bigwedge\{x/n : n\in\NN\}$ exists since $\os$ is \sDc, and it satisfies $\bigwedge\{x/n : n\in\NN\}=\bigwedge\{x/(2n) : n\in\NN\}=1/2 \bigwedge\{x/n : n\in\NN\}$. Hence $\bigwedge\{x/n : n\in\NN\}=0$. We shall use this automatic Ar\-chi\-me\-dean property at a few essential moments; see the proof of \cref{1_res:key_lemma_for_well-definednes_of_the_integral}, for example.

\medskip

 We record the following analogue of the well-known inequality for the limit inferior and the limit superior of a sequence of real numbers. We shall need it in the proof of the dominated convergence theorem; see \cref{1_res:dominated_convergence_theorem}. The proof is completely analogous to the proof for the case where $\os=\RR$.

\begin{lemma}\label{1_res:liminf_and_limsup}
Let $\os$ be a \sDc\  partially ordered vector space, and let $\seq{x}$ be an order bounded sequence in $\os$. Then $\bigvee_{n=1}^\infty\bigwedge_{k=n}^\infty x_k$ and $\bigwedge_{n=1}^\infty\bigvee_{k=n}^\infty x_k$ exist in $\os$, and
\[
\bigvee_{n=1}^\infty\bigwedge_{k=n}^\infty x_k\leq\bigwedge_{n=1}^\infty\bigvee_{k=n}^\infty x_k.
\]
\end{lemma}

In \cite{de_jeu_jiang:2021b}, we shall, amongst others, consider Riesz representation theorems for positive operators $\posmap:\contcts\to \os$; here $\ts$ is a locally compact Hausdorff space and $\os$ is a partially ordered vector space. Once could hope that, ideally, such theorems would state that this operator $\posmap$ is given by integration of a scalar-valued function with respect to an $\posos$-valued measure. However, as the case where $\os=\RR$ and $\posmap(f)\coloneqq\int_\RR\! f\di{x}$ already shows, we cannot expect this measure to be actually finite. To be able to develop a theory of measure and integration that incorporates this inevitable phenomenon, we need to adjoin an element $\largest$ to $\posos$ and let $\posR$ act on the augmented structure.  As is also done in, e.g., \cite{wright:1969,wright:1971a}, we shall actually adjoin $\largest$  to the whole space $\os$, which is necessary for a formulation of some of the results; see part~\ref{1_part:all_to_be_expected_4} of \cref{1_res:all_to_be_expected}), for example. The construction is as follows.

Firstly, we let $\osext\coloneqq\os\cup\{\largest\}$ be a disjoint union, and we extend the partial ordering from $\os$ to  $\osext$ by declaring that $x\leq \largest$ for all $x\in\osext$. The elements of $\osext$ that are in $\os$ will be called \emph{finite}. We set $\pososext\coloneqq\posos\cup\{\largest\}$. Then $\pososext$ is the set of positive elements of $\osext$.  Secondly, we make $\osext$ into an abelian additive monoid by defining $\largest+x\coloneqq\largest$ and $x+\largest\coloneqq\largest$ for all $x\in \osext$; then $\pososext$ is a sub-monoid of $\osext$. Thirdly, we define $r\cdot\largest\coloneqq\largest$ for all $r\in\posR\setminus\{0\}$, and define $0\cdot\largest\coloneqq 0$. Thus the additive monoid $\posR$ and the  multiplicative monoid $\posR$ both act as monoid homomorphisms on $\pososext$.
It is easily checked that, when $x,y\in\osext$ are such that $x\leq y$, then $x+z\leq y+z$ for all $z\in\osext$, and that $rx\leq sy$ for all $r,\,s\in\posR$ such that $r\leq s$.

We have now carried out the desired construction, but one can also go further, as follows. The construction in the previous paragraph can be applied to $\RR$. The set of positive elements of $\Rext$ is then the familiar extended positive real half line $\posRext$. It is an abelian additive monoid. The action of $\posR$ on $\pososext$ can then be extended to an action of $\posRext$ on $\pososext$ by defining $\largest\cdot 0\coloneqq 0$ and $\infty\cdot x\coloneqq\largest$ for all $x\in\pososext\setminus\{0\}$. Then the additive monoid $\posRext$ acts as monoid homomorphisms on $\pososext$. If $x,y\in\pososext$ are such that $x\leq y$ and $r,\,s\in\posRext$ are such that $r\leq s$, then $rx\leq sy$.

The construction in the previous step can be applied with $\os=\RR$, which yields an action of $\posRext$ on itself. Thus the familiar multiplicative structure on the extended positive real half line is obtained.  It is compatible with the action of $\posRext$ on $\pososext$: if $r,\,s\in\posRext$ and $x\in\pososext$, then $r\cdot(s\cdot x)=(rs)\cdot x$. Hence the multiplicative monoid $\posRext$ also acts as monoid homomorphisms on $\pososext$.

\medskip

 We shall employ the usual notation in which $a_\lambda\uparrow$ means that $\net{a}$ is an increasing net in $\os$ (or in $\osext$), and in which $a_\lambda\uparrow x$ means that $\net{a}$ is an increasing net in $\os$ (or in $\osext$) with supremum $x$ in $\os$ (or in $\osext$). The notations $a_\lambda\downarrow$ and    $a_\lambda\downarrow x$ are similarly defined. We shall be careful to indicate explicitly whether we are working in $\os$ or in $\osext$ whenever this is necessary.

In the next two results, we collect a few technical facts that will be used repeatedly in the sequel. They are quite obvious when $\os=\RR$, but it is essential for many of the proofs in the current paper and its sequels that they are generally valid. The tedious proofs are elementary.

\begin{lemma}\label{1_res:operations_in_extended_space}
Let $\os$ be a partially ordered vector space, and let $S\subseteq \osext$ be non-empty.
\begin{enumerate}
\item\label{1_part:operations_in_extended_space_1}
If $S\subseteq \os$ and  $\bigvee\{s : s\in S\}$ exists in $\os$, then this supremum is also the supremum of $S$ in $\osext$; likewise for the infimum.
\item\label{1_part:operations_in_extended_space_2}
If $\bigvee\{s : s\in S\}$ exists in $\osext$ and is finite, then $S$ consists of finite elements, and the supremum of $S$  exists in $\os$ and equals the supremum of $S$ in $\osext$.
\item\label{1_part:operations_in_extended_space_3}
If $S\neq\{\largest\}$, then $\bigwedge\{s : s\in S\}$ exists in $\osext$ if and only if $\bigwedge\{s : s\in S\cap \os\}$ exists in $\os$. If this is the case, then these infima are equal.
\item\label{1_part:operations_in_extended_space_4}
$\bigvee\{s : s\in S\}=\infty$ in $\osext$ if and only if $S$ is not bounded from above by a finite element.
\item\label{1_part:operations_in_extended_space_5}
If $\bigvee\{s : s\in S\}$ exists in $\osext$, then, for all $x\in \osext$,  $\bigvee\{x+ s : s\in S\}$ exists in $\osext$ and equals $x+\bigvee\{s : s\in S\}$; likewise for the infimum.
\item\label{1_part:operations_in_extended_space_6}
If $\bigvee\{x+s: s\in S\}$ exists in $\osext$ for some $x\in\os$, then $\bigvee\{ s : s\in S\}$ exists in $\osext$, and $\bigvee\{x + s: s\in S\}=x+\bigvee\{s: s\in S\}$; likewise for the infimum.
\item\label{1_part:operations_in_extended_space_7}
If $\bigvee\{s : s\in S\}$ exists in $\osext$, then, for all $r\in\posR$, $\bigvee\{rs : s\in S\}$ exists in $\osext$ and equals $r\bigvee\{s : s\in S\}$; likewise for the infimum.
\end{enumerate}
\end{lemma}

\begin{lemma}\label{1_res:binary_operations_in_extended_space}
Let $\os$ be a partially ordered vector space.
\begin{enumerate}
\item\label{1_part:binary_operations_in_extended_space_1}
If $A$ and $B$ are non-empty subsets of $\osext$ such that $\bigvee\{a : a\in A\}$ and $\bigvee\{b\colon b\in B\}$ exist in $\osext$, then $\bigvee\{a+b : a\in A,b\in B\}$ exists in $\osext$ and equals $\bigvee\{a : a\in A\}+\bigvee\{b : b\in B\}$; likewise for the infima.
\item\label{1_part:binary_operations_in_extended_space_2}
If $\net{a}$ and $\net{b}\subseteq \osext$ are nets in $\osext$ and $a,b\in\osext$ are such that $\net{a}\uparrow a$ and $\net{b}\uparrow b$ in $\osext$, then $\{a_\lambda+b_\lambda\}\uparrow (a+b)$ in $\osext$; likewise for decreasing nets.
\end{enumerate}
\end{lemma}

\begin{lemma}\label{1_res:completeness_properties_of_extended_space} Let $\os$ be a partially ordered vector space.
\begin{enumerate}
\item\label{1_part:completeness_properties_of_extended_space_1}
If $\os$ is \smc\  \uppars{resp.\ \sDc}, then every in\-crea\-sing sequence in \uppars{resp.\ every at most countably infinite subset of} $\osext$ has a supremum in $\osext$. If the sequence \uppars{resp.\ set} is bounded from above by a finite element, then the supremum in $\osext$ equals the supremum in $\os$. If the sequence \uppars{resp.\ subset} is not bounded from above by a finite element, then the supremum in $\osext$ equals $\largest$.
\item\label{1_part:completeness_properties_of_extended_space_2}
If $\os$ is \smc\  \uppars{resp.\ \sDc}, then every de\-crea\-sing sequence in \uppars{resp.\ every non-empty at most countably infinite sub\-set of} $\osext$ that is bounded from below in $\osext$ has an infimum in $\osext$. If all terms of the sequence are equal to $\largest$ \uppars{resp.\ if the set equals $\{\largest\}$}, then this infimum in $\osext$ equals  $\largest$. If the sequence contains finite terms \uppars{resp.\ if the subset contains finite elements}, then the infimum in $\osext$ equals the infimum in $\os$ of the decreasing subsequence of finite terms \uppars{resp.\ the subset of finite elements of the subset}, which is bounded from below by a finite element.
\item\label{1_part:completeness_properties_of_extended_space_3}
If $\os$ is \mc\  \uppars{resp.\ \Dc}, then every increasing net in \uppars{resp.\ every non-empty subset of} $\osext$ has a supremum in $\osext$. If the set \uppars{resp.\ net} is bounded from above by a finite element, then the supremum in $\osext$ equals the supremum in $\os$. If the set \uppars{resp.\ net} is not bounded from above by a finite element, then the supremum in $\osext$ equals $\largest$.
\item\label{1_part:completeness_properties_of_extended_space_4}
If $\os$ is \mc\  \uppars{resp.\ \Dc}, then every decreasing net in  \uppars{resp.\ every non-empty subset of} $\osext$ that is bounded from below in $\osext$ has an infimum in $\osext$. If all terms of the net are equal to $\largest$ \uppars{resp.\ if the subset equals $\{\largest\}$}, then this infimum in $\osext$ equals  $\largest$. If the net contains finite terms \uppars{resp.\ if the subset contains finite elements}, then the infimum in $\osext$ equals the infimum in $\os$ of the decreasing subnet of finite terms \uppars{resp.\ the subset of finite elements of the subset}, which is bounded from below by a finite element.
\end{enumerate}
\end{lemma}

\section{Monotone complete and normal partially ordered vector spaces}
\label{1_sec:normal_monotone_complete_spaces}

\noindent In \cite{de_jeu_jiang:2021b,de_jeu_jiang:2021c,de_jeu_jiang:2021d}, we shall be concerned with positive operators  $\posmap:\contcts\to\os$ or $\posmap:\contots\to\os$, where $\ts$ is a locally compact Hausdorff space and $\os$ is a partially ordered vector space; in \cite{de_jeu_jiang:2021c}, $\os$ is even a partially ordered algebra. The goal in \cite{de_jeu_jiang:2021b} is to find an $\pososext$-valued measure on the Borel subsets of $\ts$ that represents $\posmap$ via the order integrals of the present paper. Such results will then be applied in  \cite{de_jeu_jiang:2021c,de_jeu_jiang:2021d}. It will become apparent in \cref{1_sec:positive_osext-valued_measure,1_sec:integration_with_respect_to_a_positive_os-valued_measure} that, to be able to define order integrals with respect to $\pososext$-valued measures and develop their theory at all, $\os$ needs to be at least \smc. Furthermore, it will become clear in \cite{de_jeu_jiang:2021b} that, for \smc\ spaces\textemdash so that the order integrals in the aspired representation theorems make sense to begin with\textemdash the most convenient ones for which there actually \emph{is} such a representing theorem, are the spaces that are even \mc\ and that are also normal. The latter notion will be defined below.

This section contains a number of examples of \mc\ normal spaces; see the  \cref{1_res:banach_lattice_with_oc_norm_is_normal_and_monotone_complete,1_res:V_is_normal_and_monotone_complete,1_res:L_sa_is_normal_and_monotone_complete,1_res:JBW_algebra_is_normal_and_monotone_complete,1_res:combination_result_for_normality_and_monotone_completeness}. These form a preparation for the sequels to the current paper, but they may also serve as a motivation for the work on the order integral later in this paper.

The link between our measures in \cref{1_sec:positive_osext-valued_measure} and the usual strongly $\sigma$-additive ones in the context of \cref{1_res:L_sa_is_normal_and_monotone_complete} (resp.\ part~\ref{1_part:combination_result_for_normality_and_monotone_completeness_1} of \cref{1_res:combination_result_for_normality_and_monotone_completeness}) will be established in \cref{1_res:measures_with_values_in_L_sa} (resp.\ \cref{1_res:measures_with_values_in_the_regular_operators}). Apart from this connection, the remainder of the paper is independent of the current section.

\medskip

 We start by recalling some of the usual notions and introducing the notations that we shall use.

If $\os$ and $\ostwo$ are vector spaces, then $\linearop{\os,\ostwo}$ denotes the vector space of operators from $\os$ into $\ostwo$. An operator $T\in\linearop{\os,\ostwo}$ between two partially ordered vector spaces is \emph{order bounded} if it maps order bounded subsets of $\os$ into order bounded subsets of $\ostwo$.  The order bounded operators from $\os$ into $\ostwo$ form a vector space that is denoted by $\oboundedop{\os,\ostwo}$. An operator $T\in\linearop{\os,\ostwo}$ is \emph{positive} if $T(\pos{\os})\subseteq \pos{\ostwo}$, and \emph{regular} if it is the difference of two positive operators. The regular operators from $\os$ into $\ostwo$ form a vector space that is denoted by $\regularop{\os,\ostwo}$. A positive operator is order bounded, so that $\regularop{\os,\ostwo}\subseteq\oboundedop{\os,\ostwo}\subseteq\linearop{\os,\ostwo}$. If $\os$ is directed, then these three vector spaces are all partially ordered via their positive cones $\pos{\regularop{\os,\ostwo}}$. We shall write $\odualos$ for $\oboundedop{\os,\RR}$. If $\os$ is a Banach lattice, then $\odualos$ coincides with the norm dual $\ndualos$ of $\os$.

\medskip

 Order completeness properties of partially ordered vector spaces can be inherited by spaces of operators between them. For example, if $\os$ is a directed partially ordered vector space that has the Riesz decomposition property, and if $\ostwo$ is a \Dc\  vector lattice, then the spaces $\regularop{\os,\ostwo}$ and $\oboundedop{\os,\ostwo}$ coincide and are \Dc\  vector lattices; see \cite[Theorem~1.59]{aliprantis_tourky_CONES_AND_DUALITY:2007}. In par\-ti\-cu\-lar, they are \mc\  partially ordered vector spaces. The latter statement is also a consequence of the following result. We are not aware of a reference for it, even though the type of argument in it is well known; see \cite[proof of Theorem~1.19]{aliprantis_burkinshaw_POSITIVE_OPERATORS_SPRINGER_REPRINT:2006}, for example.

\begin{proposition}\label{1_res:V_is_monotone_complete}
Let $\os$ be a directed partially ordered vector space, let $\ostwo$ be a \mc\  \uppars{resp.\ \smc} partially ordered vector space, and let $V$ be a linear subspace of $\linearop{\os,\ostwo}$ containing $\regularop{\os,\ostwo}$.

Let $\net{T}$ be an increasing net \uppars{resp.\ Let $\seq{T}$ be an increasing sequence} in $V$. Then $\net{T}$ \uppars{resp.\ $\seq{T}$} is bounded from above in $V$ if and only if, for all $x\in\pos{\os}$, $\{T_\lambda x\}_{\lambda\in\Lambda}$ \uppars{resp.\ $\{T_n x\}_{n=1}^\infty$} is bounded from above in $\ostwo$. In this case, $\net{T}$ \uppars{resp.\ $\seq{T}$} has a supremum $T$ in $V$. For $x\in\pos{\os}$, it is given by $Tx=\bigvee_{\lambda\in\Lambda}T_\lambda x$ \uppars{resp.\ $Tx=\sup_{n\geq 1}T_n x$}.

In particular, $V$ is a \mc\  \uppars{resp.\ \smc} partially ordered vector space.
\end{proposition}

\begin{proof}
We prove the statements for the \mc ness and the \smc ness of $V$ at the same time. Let $\net{T}$ be an increasing net (possibly an increasing sequence) in $V$.

If the net is bounded from above by an element $S$ of $V$, then $T_\lambda x\leq Sx$ for all $x\in\pos{\os}$, so that $\{T_\lambda x\}_{\lambda\in\Lambda}$ is bounded from above in $\ostwo$ for all $x\in\pos{\os}$.

Conversely, suppose that $\{T_\lambda x\}_{\lambda\in\Lambda}$ is bounded from above in $\ostwo$ for all $x\in\pos{\os}$. We shall show that the pointwise formula for $T$ as in the statement of the theorem actually defines an element of $V$. This is then clearly the least upper bound of the net in $V$.

Choose any $\lambda_0\in\Lambda$. Using that the nets $\net{T}$ and $\{T_\lambda x\}_{\lambda\in\Lambda}$ for $x\in\pos{\os}$ are increasing, one easily sees, by considering the net $\{T_\lambda-T_{\lambda_0}\}_{\lambda\geqslant\lambda_0}\subseteq\pos{V}$, that it is sufficient to prove that the pointwise formula defines an element of $V$ when $T_\lambda\geq 0$ for all $\lambda\in\Lambda$.

Supposing, therefore, that $\net{T}\subseteq\pos{V}$, we define $T:\pos{\os}\rightarrow\pos{\ostwo}$ as in the statement of the theorem by
\begin{equation}\label{1_eq:T_definition}
T x\coloneqq\bigvee_{\lambda\in\Lambda} T_\lambda x
\end{equation}
for $x\in \pos{\os}$. Since $\{T_\lambda x\}_{\lambda\in\Lambda}$ is increasing and bounded from above, the supremum in the right hand side of \cref{1_eq:T_definition} exists as a consequence of the pertinent completeness property of $\ostwo$. It is clear that $T(r x)=r T(x)$ for all $r\geq 0$ and $x\in\pos{\os}$. Next we show that $T$ is additive on $\pos{\os}$. Fix $x_1, x_2\in\pos{\os}$. For all $\lambda\in\Lambda$, we have
\[
T_\lambda(x_1+x_2)=T_\lambda(x_1)+T_\lambda(x_2)\leq T x_1+ T x_2,
\]
so $T(x_1+x_2)\leq Tx_1+T x_2$.
For the reverse inequality, consider arbitrary $\lambda_1,\lambda_2\in\Lambda$. Choose $\lambda_3\in\Lambda$ such that $\lambda_3\geq\lambda_1$ and $\lambda_3\geq\lambda_2$. Then, using that $\net{T}$ is increasing, we have
\[
T_{\lambda_1} x_1+T_{\lambda_2} x_2\leq T_{\lambda_3}x_1+T_{\lambda_3}x_2=T_{\lambda_3}(x_1+x_2)\leq T(x_1+x_2),
\]
which easily implies that $T(x_1+x_2)\geq T x_1+ T x_2$. Hence $T$ is additive on $\pos{\os}$.

Next, if $x\in \os$ is arbitrary, we choose $x_1,x_2\in\pos{\os}$ such that $x=x_1-x_2$, and we define $Tx\coloneqq Tx_1-Tx_2$. It is then easy to see that $T$ is well defined and that $T$ is linear, so that $T\in\linearop{\os,\ostwo}$. Since clearly $T\geq 0$, we also have $T\in\regularop{\os,\ostwo}$. Since $\regularop{\os,\ostwo}\subseteq V$, we have $T\in V$, as required.

\end{proof}

The underlying spaces of the monotone ($\sigma$-)partially ordered vector spaces of operators in \cref{1_res:V_is_monotone_complete} are themselves partially ordered vector spaces, but there also exist \mc\  partially ordered vector spaces of operators where this is no longer the case. An important class of examples is provided by the following result, which is a direct consequence of  \cite[Lemma I.6.4]{davidson_C-STAR-ALGEBRAS_BY_EXAMPLE:1996}.

\begin{proposition}\label{1_res:L_sa_is_monotone_complete} Let  $\hilbert$ be a complex Hilbert space, and let $L$ be a  strongly closed complex linear subspace of $\boundedh$.  Let $L_\sa$ be the real vector space that consists of the self-adjoint elements of $L$, supplied with the partial ordering that is inherited from the usual partial ordering of the self-adjoint elements of $\boundedh$. Then $L_\sa$ is a \mc\  partially ordered vector space.

More precisely, if $\net{T}$ is an increasing net in $L_\sa$ that is bounded from above in $L_\sa$, then $\net{T}$ converges in $\boundedh$ with respect to the strong operator topology, its limit $\SOTlim_{\lambda} T_\lambda$ is an element of $L_{\sa}$, and is equal to the supremum $\bigvee\{T_\lambda : \lambda\in\Lambda\}$ of the $T_\lambda$ in $L_{\sa}$.
\end{proposition}

\begin{remark}\label{1_rem:kadison}
In view of Kadison's anti-lattice theorem \cite[Theorem~6]{kadison:1951}, the space $L_\sa$ figuring in \cref{1_res:L_sa_is_monotone_complete} will not generally be a vector lattice. For example, if $\dim \hilbert\geq 2$, then, by Kadison's result, $\boundedh_{\sa}$ is not a vector lattice. The space $\boundedh_\sa$ for $\dim \hilbert\geq 2$ also provides an example of a \mc\  partially ordered vector space that is not \Dc. It is, in fact, not even \sDc. To see this, suppose that it is \sDc. Since $\boundedh_\sa$ is directed, every subset $\{T_1,T_2\}$ is bounded from above, so that it then has a supremum. Thus $\boundedh_\sa$ is a vector lattice, but we know this not to be the case.
\end{remark}

Now that the monotone completeness of the above spaces of operators has been established, we turn to normality. As a preparation, we start with the following definition.

\begin{definition}\label{1_def:order_continuity} Let $\os$ and $\ostwo$ be partially ordered vector spaces, and let $T:\os\to \ostwo$ be a positive operator. Then $T$ is called \emph{order continuous} (resp.\ \emph{$\sigma$-order continuous}) if $Tx_\lambda\downarrow 0$ in $\ostwo$ whenever $x_\lambda\downarrow 0$ in $\os$ (resp.\ if $Tx_n\downarrow 0$ in $\ostwo$  whenever $x_n\downarrow 0$ in $\os$). One can, equivalently, require that  $Tx_\lambda\uparrow Tx$ in $\ostwo$ whenever $0\leq x_\lambda\uparrow x$ in $\os$ (resp.\ that $Tx_n\uparrow Tx$ in $\ostwo$  whenever $0\leq x_n\uparrow x$ in $\os$).  A general  operator in $\regularop{\os,\ostwo}$ is order continuous (resp.\ $\sigma$-order continuous) if it is the difference of two positive order continuous operators. We shall write $\ocontop{\os,\ostwo}$ (resp.\ $\socontop{\os,\ostwo}$) for the order continuous (resp.\ $\sigma$-order continuous) operators from $\os$ into $\ostwo$.
\end{definition}

It is easy to see that the sum of two positive order continuous (resp.\ $\sigma$-order continuous) operators is again a positive order continuous (resp.\ positive $\sigma$-order continuous) operator, and it follows that $\ocontop{\os,\ostwo}$ (resp.\ $\socontop{\os,\ostwo}$) is a linear subspace of $\regularop{\os,\ostwo}$. If $\os$ is directed, then it is a partially ordered vector space with the positive order continuous operators (resp.\ the positive  $\sigma$-order continuous operators) as its positive cone, which is generating by definition.
We shall write $\ocdualos$ for $\ocontop{\os,\RR}$ and $\socdualos$ for $\socontop{\os,\RR}$. It is the space $\ocdualos$ that will be of help in the context of Riesz representation theorems for vector-valued positive maps. It has $\pos{(\ocdualos)}$, the positive order continuous functionals, as its generating positive cone.

\begin{remark}\label{1_rem:order_continuity_agrees}
If $\os$ and $\ostwo$ are vector lattices, where $\ostwo$ is \Dc, then the above notion of order continuous operators agrees with the usual one in the literature. To see this, recall that $T:\os\to \ostwo$ is order continuous in the sense of \cite[p.~123]{zaanen_RIESZ_SPACES_VOLUME_II:1983} if $\abs{Tx_\lambda}\downarrow 0$ in $\ostwo$ whenever $x_\lambda\downarrow 0$ in $\os$. Hence a positive $T$ is order continuous in the sense of \cite[p.~123]{zaanen_RIESZ_SPACES_VOLUME_II:1983} precisely when it is order continuous in the sense of our \cref{1_def:order_continuity}. Furthermore, by \cite[Lemma~84.1]{zaanen_RIESZ_SPACES_VOLUME_II:1983}, $T: \os\to \ostwo$ is order continuous in the sense of \cite[p.~123]{zaanen_RIESZ_SPACES_VOLUME_II:1983} if and only if $\pos{T}$ and $\negt{T}$ are order continuous in the sense of that same definition, i.e.,  if and only if they are positive order continuous operators in the sense of \cref{1_def:order_continuity}. In addition, \cite[Theorem~84.2]{zaanen_RIESZ_SPACES_VOLUME_II:1983} implies that the set of all $T: \os\to \ostwo$ that are order continuous in the sense of \cite[p.~123]{zaanen_RIESZ_SPACES_VOLUME_II:1983} form a vector space. It follows that the notions of (and notations for) order continuous operators coincide if $\os$ and $\ostwo$ are vector lattices, where $\ostwo$ is \Dc. A similar argument shows that this is also the case for $\sigma$-order continuous operators.
\end{remark}

\begin{definition}\label{1_def:normal_space}
Let $\os$ be a partially ordered vector space. Then $\os$ is called \emph{normal} when, for $x\in \os$, $\f{x,x^\prime}\geq 0$ for all $x^\prime\in\pos{(\ocdualos)}$ if and only if $x\in \pos{\os}$. We say that $\os$ is \emph{$\sigma$-normal} when, for $x\in \os$, $\f{x,x^\prime}\geq 0$ for all $x^\prime\in\pos{(\socdualos)}$ if and only if $x\in \pos{\os}$.
\end{definition}

Clearly, if $\os$ is ($\sigma$-)normal, then $\pos{(\ocdualos)}$ separates the points of $\os$.

As with order continuous operators, if $\os$ is a vector lattice, then our notion of normality coincides with that in the literature. To see this, we include the following result.

\begin{lemma}\label{1_res:separating_properties}
Let $\os$ be a vector lattice, and let $A$ be an order ideal of $\odualos$. Then the following are equivalent:
\begin{enumerate}
\item\label{1_part:separating_properties_1}
$\pos{A}$ separates the points of $\os$;
\item\label{1_part:separating_properties_2}
$A$ separates the points of $\os$;
\item\label{1_part:separating_properties_3}
For $x\in \os$, $\f{x,x^\prime}\geq 0$ for all $x^\prime\in \pos{A}$ if and only if $x\in \pos{\os}$.
\end{enumerate}
\end{lemma}

If $A=\odualos$, then the fact that part~\ref{1_part:separating_properties_2} implies part~\ref{1_part:separating_properties_3} can be found as \cite[Theorem~1.66]{aliprantis_burkinshaw_POSITIVE_OPERATORS_SPRINGER_REPRINT:2006}). The following proof is an adaptation of the proof for that case.

\begin{proof}
It is clear that part~\ref{1_part:separating_properties_1} implies part~\ref{1_part:separating_properties_2}, and also that part~\ref{1_part:separating_properties_3} implies part~\ref{1_part:separating_properties_1}.
It remains to be shown that part~\ref{1_part:separating_properties_2} implies part~\ref{1_part:separating_properties_3}.

One implication in part~\ref{1_part:separating_properties_3} is trivial, so we turn to the non-trivial one.
Let $x\in \os$ be such that $\f{x,x^\prime} \geq 0$ for all $x^\prime\in \pos{A}$. If $x^\prime\in\pos{A}$, then \cite[Theorem~1.23]{aliprantis_burkinshaw_POSITIVE_OPERATORS_SPRINGER_REPRINT:2006} shows that there exists $y^\prime\in\odualos$ such that $0\leq y^\prime\leq x^\prime$ and
\[
\f{\negt{x},x^\prime}=-\f{x,y^\prime},
\]
Since $A$ is an order ideal, we have $y^\prime \in\pos{A}$. Hence $\f{x,y^\prime}\geq 0$ by assumption, and we see that $\f{\negt{x},x^\prime}\leq 0$. On the other hand, it is clear that $\f{\negt{x},x^\prime}\geq 0$. We conclude that $\f{\negt{x},x^\prime}=0$ for all $x^\prime\in\pos{A}$, so that $\f{\negt{x},x^\prime}=0$ for all $x^\prime\in A $. Then $\negt{x}=0$ by the separation property of~$A$, and therefore $x\in \pos{\os}$.
\end{proof}

As a consequence, a vector lattice is normal in the sense of our \cref{1_def:normal_space} if and only if $\ocdualos$ separates the points of $\os$, i.e.,  if and only if it is normal in the sense of \cite[p.~21]{abramovich_aliprantis_INVITATION_TO_OPERATOR_THEORY:2002}. For vector lattices, therefore, our notion of normality coincides with the one in the literature.

\medskip

As will become clear in \cite{de_jeu_jiang:2021b}, the importance of normality for our work on Riesz representation theorems lies in the following observation.

\begin{proposition}\label{1_res:inf_and_sup_via_order_dual}
Let $\os$ be a normal partially ordered vector space. Suppose that $\net{x}$ is a net in $\os$, and that $x\in\os$.
\begin{enumerate}
\item\label{1_part:inf_and_sup_via_order_dual_1}
If $x_\lambda\downarrow$, then $x_\lambda\downarrow x$ if and only if $\f{x,x^\prime}=\inf_{\lambda\in\Lambda }\f{x_\lambda,x^\prime}$ for all $x^\prime\in\pos{(\ocdualos)}$.
\item\label{1_part:inf_and_sup_via_order_dual_2}
If $x_\lambda\uparrow$, then $x_\lambda\uparrow x$ if and only if $\f{x,x^\prime}=\sup_{\lambda\in\Lambda }\f{x_\lambda,x^\prime}$ for all $x^\prime\in\pos{(\ocdualos)}$.
\end{enumerate}
\end{proposition}

\begin{proof}
We prove part~\ref{1_part:inf_and_sup_via_order_dual_1} where $x_\lambda\downarrow$; part~\ref{1_part:inf_and_sup_via_order_dual_2} follows from this.

If $x_\lambda\downarrow x$, so that $x_\lambda -x\downarrow 0$, and if $x^\prime\in\pos{(\ocdualos)}$, then, by definition, $\f{x_\lambda -x,x^\prime}\downarrow 0$. Hence $\f{x_\lambda,x^\prime}\downarrow\f{x,x^\prime}$, so that, in particular,  $\f{x,x^\prime}=\inf_{\lambda\in\Lambda }\f{x_\lambda,x^\prime}$.

Conversely, suppose that $\f{x,x^\prime}=\inf_{\lambda\in\Lambda }\f{x_\lambda,x^\prime}$ for all $x^\prime\in\pos{(\ocdualos)}$.

With $\lambda\in\Lambda$ fixed, we then have $\f{x,x^\prime}\leq\f{x_\lambda,x^\prime}$, or $\f{x-x_\lambda,x^\prime}\leq 0$, for all $x^\prime\in\pos{(\ocdualos)}$. Since $\os$ is normal, we conclude that $x\leq x_\lambda$. Hence $x$ is a  lower bound of $\{x_\lambda : \lambda\in\Lambda\}$.

If $\tilde x$ is a lower bound of $\{x_\lambda : \lambda\in\Lambda\}$, and if $x^\prime\in \pos{(\ocdualos)}$ is fixed, then certainly $\f{\tilde x,x^\prime}\leq\f{x_\lambda,x^\prime}$ for all $\lambda\in\Lambda$. Hence $\f{\tilde x,x^\prime}\leq\inf_{\lambda\in\Lambda }\f{x_\lambda,x^\prime}$. Since the right hand side of this inequality equals $\f{x,x^\prime}$ by assumption, we have $\f{\tilde x,x^\prime}\leq \f{x,x^\prime}$ for all $x^\prime\in\pos{(\ocdualos)}$. Since $\os$ is normal, we see that $\tilde x\leq x$.

We conclude that $x=\bigwedge \{x_\lambda :\lambda\in\Lambda\}$.  Hence $x_\lambda\downarrow x$.
\end{proof}

A similar proof establishes the following.

\begin{proposition}\label{1_res:the_sigma_property_for_order dual}
	Let $\os$ be a $\sigma$-normal partially ordered vector space. Suppose that $\seq{x}$ is a sequence in $\os$, and that $x\in\os$.
	\begin{enumerate}
		\item\label{1_part:the_sigma_property_for_order dual_1}
		If $x_n\downarrow$, then $x_n\downarrow x$ if and only if $\f{x,x^\prime}=\inf_{n\geq 1}\f{x_n,x^\prime}$ for all $x^\prime\in\pos{(\socdualos)}$.\label{1_part:sigma_inf}
		\item\label{1_part:the_sigma_property_for_order dual_2}
		If $x_n\uparrow$, then $x_n\uparrow x$ if and only if $\f{x,x^\prime}=\sup_{n\geq 1}\f{x_n,x^\prime}$ for all $x^\prime\in\pos{(\socdualos)}$.\label{1_part:sigma_sup}
	\end{enumerate}
\end{proposition}

As mentioned in the introduction of this section, the partially ordered vector spaces that are both \mc\ normal are a convenient context for Riesz representation theorems in terms of order integrals. We now include a few examples.

\begin{proposition}\label{1_res:banach_lattice_with_oc_norm_is_normal_and_monotone_complete}
A Banach lattice with an order continuous norm is a \mc\ and normal partially ordered vector space.
\end{proposition}

\begin{proof}
	Let $\os$ be a Banach lattice with an order continuous norm. Then  $\os$ is \Dc;  see \cite[Corollary~4.10]{aliprantis_burkinshaw_POSITIVE_OPERATORS_SPRINGER_REPRINT:2006}, for example.. Hence $\os$ is certainly \mc. It follows easily from the fact that $\odualos=\ndualos$ and the order continuity of the norm that $\odualos=\ocdualos$. Hence $\ocdualos=\ndualos$. Since $\ndualos$ separates the points of $\os$, \cref{1_res:separating_properties} shows that $\os$ is normal.
\end{proof}

\begin{proposition}\label{1_res:V_is_normal_and_monotone_complete}
Let $\os$ be a directed partially ordered vector space, let $\ostwo$ be a \mc\ and normal partially ordered vector space, and let $V$ be a linear subspace of $\linearop{\os,\ostwo}$ that contains $\regularop{\os,\ostwo}$.

Then $V$ is a \mc\ and normal partially ordered vector space.
\end{proposition}

\begin{proof}
\cref{1_res:V_is_monotone_complete} shows that $V$ is \mc. To prove that it is normal, we define, for $x\in\pos{\os}$ and $x^\prime\in\pos{(\ocdual{\ostwo})}$, the functional $\varphi_{x,x^\prime}:V\to\RR$ by setting $\f{T,\varphi_{x,x^\prime}}\coloneqq\f{Tx,x^\prime}$ for $T\in V$. Then $\varphi_{x,x^\prime}$ is evidently positive. We claim that it is order continuous. To see this, fix $x\in\os$ and $x^\prime\in\pos{(\ocdual{\ostwo})}$, and suppose that $T_\lambda\downarrow 0$ in $V$. By \cref{1_res:V_is_monotone_complete}, this implies that $T_\lambda x\downarrow 0$. Since $x^\prime$ is order continuous, it then follows that $\f{T_\lambda x,x^\prime}\downarrow 0$. Hence $\varphi_{x,x^\prime}$ is  order continuous, and we conclude that $\varphi_{x,x^\prime}\in\pos{(\ocdual{V})}$. Finally, suppose that $T\in V$ is such that $\f{T,\varphi_{x,x^\prime}}\geq 0$ for all $x\in\pos{\os}$ and $x^\prime\in\pos{(\ocdual{\ostwo})}$. Since $\ostwo$ is normal, this implies that $Tx\geq 0$ for all $x\in\pos{\os}$. Hence $T\geq 0$. This proves that $V$ is normal.
\end{proof}

The following example is a continuation of \cref{1_res:L_sa_is_monotone_complete}.

\begin{proposition}\label{1_res:L_sa_is_normal_and_monotone_complete} Let  $\hilbert$ be a complex Hilbert space, and let $L$ be a strongly closed complex linear subspace of $\boundedh$.  Let $L_\sa$ be the real vector space that consists of the self-adjoint elements of $L$, supplied with the partial ordering that is inherited from the usual partial ordering on $\boundedh_\sa$. Then $L_\sa$ is a \mc\  and normal partially ordered vector space.
\end{proposition}

\begin{proof}
In view of \cref{1_res:L_sa_is_monotone_complete}, we need to prove only that $L_\sa$ is normal. To this end, we define, for $x\in \hilbert$, the functional $\varphi_x:L_\sa\to\RR$ by $\f{T,\varphi_x}=\inp{Tx,x}$ for $T\in L_\sa$; here $\inp{\,\cdot\,,\cdot\,}$ denotes the inner product on $\hilbert$.  Then $\varphi_x$ is clearly positive. We claim that $\varphi_x$ is order continuous. To see this, suppose that $T_\lambda\downarrow 0$. By \cref{1_res:L_sa_is_monotone_complete}, this implies that $\SOTlim_{\lambda}T_\lambda=0$. Consequently,  we have $\f{T_\lambda,\varphi_x}=\inp{T_\lambda x,x}\downarrow 0$. Hence $\varphi_x$ is order continuous, and we conclude that $\varphi_x\in\pos{(\ocdual{V})}$. Finally, suppose that $T\in L_\sa$ is such that $\f{T,\varphi_x}\geq 0$ for all $x\in \hilbert$. Hence $\inp{Tx,x}\geq 0$ for all $x\in \hilbert$, so that $T\geq 0$. This proves that $L_\sa$ is normal.
\end{proof}

As von Neumann algebras are canonical examples of JBW-algebras (see, e.g.,  \cite[Definition~2.2]{alfsen_shultz_GEOMETRY_OF_STATE_SPACES_OF_OPERATOR_ALGEBRAS:2003} for a definition of the latter), the next example is somewhat related to \cref{1_res:L_sa_is_normal_and_monotone_complete}. We recall that any JB-algebra is partially ordered by its cone of squares (see \cite[Lemma~1.10]{alfsen_shultz_GEOMETRY_OF_STATE_SPACES_OF_OPERATOR_ALGEBRAS:2003}) and that this cone is generating; see \cite[Proposition~1.28]{alfsen_shultz_GEOMETRY_OF_STATE_SPACES_OF_OPERATOR_ALGEBRAS:2003}. The following is immediate from \cite[Corollary~2.17]{alfsen_shultz_GEOMETRY_OF_STATE_SPACES_OF_OPERATOR_ALGEBRAS:2003}.

\begin{proposition}\label{1_res:JBW_algebra_is_normal_and_monotone_complete}
A JBW-algebra is a normal and \mc\ partially ordered vector space.
\end{proposition}

The \cref{1_res:banach_lattice_with_oc_norm_is_normal_and_monotone_complete,1_res:V_is_normal_and_monotone_complete,1_res:L_sa_is_normal_and_monotone_complete,1_res:JBW_algebra_is_normal_and_monotone_complete} can be combined in various ways. The following immediate result seems worth recording explicitly.

\begin{theorem}\label{1_res:combination_result_for_normality_and_monotone_completeness}
Let $\os$ be a directed partially ordered vector space.

\begin{enumerate}
\item\label{1_part:combination_result_for_normality_and_monotone_completeness_1}
If $\ostwo$ is a Banach lattice with an order continuous norm, and if $V$ is a linear subspace of $\linearop{\os,\ostwo}$ that contains $\regularop{\os,\ostwo}$, then $V$ is a \mc\ and normal partially ordered vector space. In particular, $\regularop{\os}$ is a \Dc\ and normal vector lattice for every Banach lattice $\os$ with an order continuous norm.
\item\label{1_part:combination_result_for_normality_and_monotone_completeness_2}
If $L$ is a strongly closed complex linear subspace of the bounded operators on a complex Hilbert space, with self-adjoint part $L_\sa$, and if $V$ is a linear subspace of $\linearop{\os,L_\sa}$ that contains $\regularop{\os,L_\sa}$, then $V$ is a \mc\ and normal partially ordered vector space. In particular, $\regularop{L_\sa}$ is a \mc\ and normal partially ordered vector space.
\item\label{1_part:combination_result_for_normality_and_monotone_completeness_3}
If $\jbw$ is a JBW-algebra, and if $V$ is a linear subspace of $\linearop{\os,\jbw}$ that contains $\regularop{\os,\jbw}$, then $V$ is a \mc\ and normal partially ordered vector space. In particular, $\regularop{\jbw}$ is a \mc\ and normal  partially ordered vector space.
\end{enumerate}
\end{theorem}

\section{$\pososext$-valued measures}
\label{1_sec:positive_osext-valued_measure}

\noindent After the motivational \cref{1_sec:normal_monotone_complete_spaces}, we now start with the actual measure and integration theory. The current section is concerned with the basics for measures with values in the extended positive cone $\pososext$ of a partially ordered vector space $\os$.

In this paper, a \emph{measurable space} is a pair $\ms$, where $\pset$ is a set and $\alg$ is an algebra of subsets of $\pset$; that is,  $\alg$ is a non-empty collection of subsets of $\pset$ that is closed under the taking of complements and under the taking of finite unions. The elements of $\alg$ are the measurable subsets of $\pset$. It will become necessary only in \cref{1_sec:integration_with_respect_to_a_positive_os-valued_measure} to suppose that $\alg$ is a $\sigma$-algebra.

The $\sigma$-additivity of a measure $\npm:\alg\to\posRext$ requires that $\npm\left(\bigcup_{n=1}^\infty\mss_n\right)=\sum_{n=1}^\infty\npm(\mss_n)$ in $\posRext$ whenever $\seq{\mss}$ is pairwise disjoint sequence in $\alg$ such that $\bigcup_{n=1}^\infty\mss_n\in\alg$.\footnote{In other sources, such $\npm$ can then be called a pre-measure; see \cite[Definition~I.3.1]{bauer_MEASURE_AND_INTEGRATION_THEORY:2001}, for example.} In this definition, the convergence of the series in $\posRext$ has to be given a meaning. One possibility is to interpret it as the convergence of the sequence of partial sums in the topology of $\posRext$ as the one-point compactification of $\RR$. This does not admit a generalisation to maps $\npm:\alg\to\pososext$ for a general partially ordered vector space $\os$, since no topology need be present. If $\npm$ takes values in $\posRext$, then one can, however, equivalently require that $\npm\left(\bigcup_{n=1}^\infty\mss_n\right)=\bigvee_{N=1}^\infty\sum_{n=1}^N\npm(\mss_n)$ in $\posRext$, where the supremum is to be taken in the partially ordered set $\posRext$. Since this involves only finite sums, topological convergence is no longer an issue and the requirement \emph{does} make sense for general $\os$, provided one guarantees that the supremum in the right hand side always exists. Thus one is led to suppose that $\os$ be \smc.

The following definition is, therefore, a natural one.  It is due to Wright; see \cite[p.~111]{wright:1969}. We recall that the extension $\osext$ of a partially ordered vector space $\os$ has been introduced in \cref{1_sec:partially_ordered_vector_spaces}.

\begin{definition}\label{1_def:positive_pososext_valued_measure}
Let $\ms$ be a measurable space, and let $\os$ be a \smc\  partially ordered vector space. An \emph{$\pososext$-valued measure on $\alg$} is a map $\npm:\alg\rightarrow \pososext$ such that:
\begin{enumerate}
\item\label{1_part:pososext_valued_measure_1}
$\npm(\emptyset)=0$;
\item\label{1_part:pososext_valued_measure_2}
whenever $\seq{\mss}$ is a pairwise disjoint sequence in $\alg$ with $\bigcup_{n=1}^\infty\mss_n\in\alg$, then
 \begin{equation}\label{1_eq:sigma_additivity}
 \npm\left(\bigcup_{n=1}^\infty\mss_n\right)=\bigvee_{N=1}^\infty\sum_{n=1}^N\npm(\mss_n)
 \end{equation}
 in $\osext$.
\end{enumerate}

Since the partial sums form a increasing sequence in $\osext$ because $\npm(\alg)\subseteq \pososext$, it follows from part~\ref{1_part:completeness_properties_of_extended_space_1} of \cref{1_res:completeness_properties_of_extended_space} that, for a given enumeration of the $\mss_n$, the right hand side of \cref{1_eq:sigma_additivity} always exists in $\osext$. A moment's thought shows that this supremum is actually independent of the choice for the enumeration.
\end{definition}

A quadruple $\msm$, where $\pset$ is a set, $\alg$ is an algebra of subsets of $\pset$, $\os$ is a \smc\ partially ordered vector space, and $\npm:\alg\to\pososext$ is an $\pososext$-valued measure on $\alg$, will be called a \emph{measure space}. We shall refer to the property under part~\ref{1_part:pososext_valued_measure_2} of \cref{1_def:positive_pososext_valued_measure} as the \emph{$\sigma$-additivity of $\npm$}.

If $\npm(\pset)\in\posos$, then we say that $\npm$ is \emph{finite}, or that it is \emph{$\posos$-valued}. It will follow from part~\ref{1_part:all_to_be_expected_2} of \cref{1_res:all_to_be_expected} that then $\npm(\mss)\in\posos$ for all $\mss\in\alg$. If $\npm$ is not finite, i.e.,  if $\npm(\ts)=\largest$, then $\mu$ is said to be \emph{infinite}.

A measurable subset $\mss$ of $\pset$ is \emph{$\sigma$-finite} if there exists a sequence $\seq{\mss}$ in $\alg$ such that $\mss\subseteq\bigcup_{n=1}^\infty \mss_n$ and $\npm(\mss_n)\in\posos$ for all $n\geq 1$. If $\pset$ is $\sigma$-finite, then we say that $\npm$ is a \emph{$\sigma$-finite measure}.

A \emph{null set} is a measurable subset of $\pset$ with measure zero. A measure space is called  \emph{complete} if a subset of a null set is still a measurable subset. It will follow from part~\ref{1_part:all_to_be_expected_2} of \cref{1_res:all_to_be_expected} that a measurable subset of a null set is again a null set.

A property that a point $\pt$ in $\pset$ may or may not have is said to hold \emph{$\npm$-almost everywhere}, or to hold \emph{for $\npm$-almost all $\pt$ in $\pset$}, if the subset of $\pset$ consisting of those points that do not have this property is contained in a null set. It is not required that this subset of exceptional points be measurable. If the measure is clear from the context, we shall simply  write that the property holds almost everywhere, or that it holds for almost all $\pt$ in $\pset$.

\medskip

 Before we proceed with the general theory, let us, by way of motivation, consider two particular cases that we have in mind if $\os$ happens to be a space of operators. They show that our ordered requirement for the $\sigma$-additivity  of an operator-valued measure is then the same as the classical $\sigma$-additivity in the strong operator topology.

 The first case is in the context of \cref{1_res:L_sa_is_normal_and_monotone_complete}.

\begin{lemma}\label{1_res:measures_with_values_in_L_sa}
	Let  $\hilbert$ be a complex Hilbert space, and let $L$ be a strongly closed complex linear subspace of $\boundedh$.  Let $L_\sa$ be the real vector space that consists of the self-adjoint elements of $L$, supplied with the partial ordering that is inherited from the usual partial ordering on $\boundedh_\sa$.
	
	Let $\ms$ be a measurable space, let $\npm:\alg\to\pos{L_\sa}$ be a map such that $\npm(\emptyset)=0$, and let $\seq{\mss}\subseteq\alg$ be such that $\bigcup_{n=1}^\infty\mss_n\in\alg$.
	
	Then the following are equivalent:
	\begin{enumerate}
	\item\label{1_part:measures_with_values_in_L_sa_1} $\npm\left(\bigcup_{n=1}^\infty\mss_n\right)=\bigvee_{N=1}^\infty\sum_{n=1}^ N\npm(\mss_n)$ in $L_\sa$;
	\item\label{1_part:measures_with_values_in_L_sa_2}
	$\npm\left(\bigcup_{n=1}^\infty\mss_n\right)x=\sum_{n=1}^\infty \npm(\mss_n) x$ in the norm topology of $\hilbert$ for all $x\in \hilbert$.
	\end{enumerate}
	
\end{lemma}

\begin{proof}
It is immediate from \cref{1_res:L_sa_is_monotone_complete} that  part~\ref{1_part:measures_with_values_in_L_sa_1} implies part~\ref{1_part:measures_with_values_in_L_sa_2}.
	
We prove that part~\ref{1_part:measures_with_values_in_L_sa_2} implies part~\ref{1_part:measures_with_values_in_L_sa_1}. For each $N\geq 1$, we have, for all $x\in\hilbert$,
\begin{align*}
\lrinp{\sum_{n=1}^N\npm(\mss_n)x,x}&=\sum_{n=1}^N\lrinp{\npm(\mss_n)x,x}\\
&\leq \sum_{n=1}^\infty\lrinp{\npm(\mss_n)x,x}\\
&= \lrinp{\sum_{n=1}^\infty\npm(\mss_n)x,x}\\
&=\lrinp{\npm\left(\bigcup_{n=1}^\infty\mss_n\right)x,x}.
\end{align*}
We thus see that $\sum_{n=1}^N\npm(\mss_n)\leq\npm\left(\bigcup_{n=1}^\infty\mss_n\right)$ for all $N\geq 1$. Since $\sum_{n=1}^N\npm(\mss_n)\uparrow$,  \cref{1_res:L_sa_is_monotone_complete} shows that the sequence has a supremum in $L_\sa$, and that this supremum is also its SOT-limit. It is obvious from the validity of part~\ref{1_part:measures_with_values_in_L_sa_2} that this SOT-limit is $\npm\left(\bigcup_{n=1}^\infty\mss_n\right)$.
\end{proof}

The second case is in the context of part~\ref{1_part:combination_result_for_normality_and_monotone_completeness_1} of \cref{1_res:combination_result_for_normality_and_monotone_completeness}.

\begin{lemma}\label{1_res:measures_with_values_in_the_regular_operators}
	Let $\ms$ be a measurable space, let $\os$ be Dedekind complete Banach lattice, let $\npm:\alg\to\pos{\regularop{\os}}$ be a map such that $\npm(\emptyset)=0$, and let $\seq{\mss}\subseteq\alg$ be such that $\bigcup_{n=1}^\infty\mss_n\in\alg$.
	
	\begin{enumerate}
		\item\label{1_part:measures_with_values_in_the_regular_operators_1}
		If $\npm\left(\bigcup_{n=1}^\infty\mss_n\right)x=\sum_{n=1}^\infty \npm(\mss_n) x$ in the norm topology on $\os$ for all $x\in \os$, then $\npm\left(\bigcup_{n=1}^\infty\mss_n\right)=\bigvee_{N=1}^\infty\sum_{n=1}^ N\npm(\mss_n)$ in $\regularop{\os}$.
		\item\label{1_part:measures_with_values_in_the_regular_operators_2}
		If $\npm\left(\bigcup_{n=1}^\infty\mss_n\right)=\bigvee_{N=1}^\infty\sum_{n=1}^ N\npm(\mss_n)$ in $\regularop{\os}$ and if the norm on $\os$ is $\sigma$-order continuous,\footnote{The combination of the $\sigma$-order continuity of the norm and the ($\sigma$-)\Dc ness implies that the norm is even order continuous; see \cite[Theorem~2.4.2]{meyer-nieberg_BANACH_LATTICES:1991} } then $\npm\left(\bigcup_{n=1}^\infty\mss_n\right)x=\sum_{n=1}^\infty \npm(\mss_n) x$ in the norm topology of $\os$ for all $x\in \os$.
	\end{enumerate}
\end{lemma}

\begin{proof}
	We prove part~\ref{1_part:measures_with_values_in_the_regular_operators_1}. If $x\in\posos$, then $\sum_{n=1}^N\npm(\mss_n)x\uparrow$. Since this increasing sequence is norm convergent, its norm limit is also its supremum, i.e.,
	$\sum_{n=1}^N\npm(\mss_n)x\uparrow \sum_{n=1}^\infty \npm(\mss_n) x =\npm\left(\bigcup_{n=1}^\infty\mss_n\right)x$. This shows that $\sum_{n=1}^N\npm(\mss_n)\uparrow\npm\left(\bigcup_{n=1}^\infty\mss_n\right)$.
	
	We prove part~\ref{1_part:measures_with_values_in_the_regular_operators_2}. Let $x\in\posos$. Since $\sum_{n=1}^N\npm(\mss_n)\uparrow \npm\left(\bigcup_{n=1}^\infty\mss_n\right)$, we also have $\sum_{n=1}^N\npm(\mss_n)x\uparrow \npm\left(\bigcup_{n=1}^\infty\mss_n\right)x$. Since the norm is $\sigma$-order continuous, we see that $\npm\left(\bigcup_{n=1}^\infty\mss_n\right)x=\sum_{n=1}^\infty\npm(\mss_n)x$ in the norm topology of $\os$. By linearity, this is then also true for arbitrary $x\in\os$.
\end{proof}

Continuing with the general theory, we collect the usual suspects in the following result.

\begin{lemma}\label{1_res:all_to_be_expected}Let $\msm$ be a measure space.
\begin{enumerate}
\item\label{1_part:all_to_be_expected_1}
If $\mss_1,\dotsc,\mss_n\in\alg$ are pairwise disjoint, then $\npm\left(\bigcup_{i=1}^n\mss_i\right)=\sum_{i=1}^n\npm(\mss_i)$ in $\osext$.
\item\label{1_part:all_to_be_expected_2}
If $\mss_1,\mss_2\in\alg$ and $\mss_1\subseteq\mss_2$, then $\npm(\mss_1)\leq\npm(\mss_2)$ in $\osext$.
\item\label{1_part:all_to_be_expected_3}
If $\mss_1, \mss_2\in\alg$, then
$\npm(\mss_1)+\npm(\mss_2)=\npm(\mss_1\cap\mss_2)+\npm(\mss_1\cup\mss_2)$ in $\osext$.
\item\label{1_part:all_to_be_expected_4}
If $\mss_1,\mss_2\in\alg$, $\mss_1\supseteq\mss_2$, and $\npm(\mss_2)\in\os$, then $\npm(\mss_1\setminus\mss_2)=\npm(\mss_1)-\npm(\mss_2)$ in $\osext$.
\item\label{1_part:all_to_be_expected_5}
If $\seq{\mss}$ is a sequence in $\alg$ with $\bigcup_{n=1}^\infty\mss_n\in\alg$, then
 \[
 \npm\left(\bigcup_{n=1}^\infty\mss_n\right)\leq\bigvee_{N=1}^\infty\sum_{n=1}^N\npm(\mss_n),
 \]
 in $\osext$.
\end{enumerate}
\end{lemma}

\begin{proof}
Part~\ref{1_part:all_to_be_expected_1} follows from the definitions when choosing $\mss_k=\emptyset$ for $k\geq n+1$.

Part~\ref{1_part:all_to_be_expected_2} is immediate from part~\ref{1_part:all_to_be_expected_1}.

For part~\ref{1_part:all_to_be_expected_3}, we use the disjoint decomposition $\mss_1=(\mss_1\cap\mss_2)\cup(\mss_1\setminus\mss_2)$ and part~\ref{1_part:all_to_be_expected_1} to see that
\[
\npm(\mss_1)=\npm(\mss_1\cap\mss_2)+\npm(\mss_1\setminus\mss_2).
\]
Similarly, we have
\[
\npm(\mss_2)=\npm(\mss_2\cap\mss_1)+\npm(\mss_2\setminus\mss_1).
\]
Hence
\[
\npm(\mss_1)+\npm(\mss_2)=2\npm(\mss_1\cap\mss_2)+\npm(\mss_1\setminus\mss_2)+\npm(\mss_2\setminus\mss_1)
\]
in $\osext$. Since $\mss_1\cap\mss_2$, $\mss_1\setminus\mss_2$, and $\mss_2\setminus\mss_1$ are pairwise disjoint, and since their union equals $\mss_1\cup\mss_2$, part~\ref{1_part:all_to_be_expected_1} shows that the right hand side of the above equation equals $\npm(\mss_1\cap\mss_2)+\npm(\mss_1\cup\mss_2)$, as required.

For part~\ref{1_part:all_to_be_expected_4}, we use part~\ref{1_part:all_to_be_expected_1} to see that $\npm(\mss_1)=\npm(\mss_2)+\npm(\mss_1\setminus\mss_2)$.
Since $\npm(\mss_2)\in\os$, it has an inverse $-\npm(\mss_2)$ in the monoid $\osext$; adding this inverse to both sides yields that $\npm(\mss_1\setminus\mss_2)=\npm(\mss_1)-\npm(\mss_2)$, as required.

For part~\ref{1_part:all_to_be_expected_5}, we let $\widetilde\mss_1=\mss_1$ and $\widetilde\mss_n=\mss_n\setminus{\bigcup_{k=1}^{n-1}\mss_k}$ for $n\geq 2$. Then, using part~\ref{1_part:all_to_be_expected_2}, we see that
\[
\npm\left(\bigcup_{n=1}^\infty\mss_n\right)=\npm\left(\bigcup_{n=1}^\infty\widetilde\mss_n\right)=\bigvee_{N=1}^\infty \sum_{n=1}^N\npm(\widetilde\mss_n)\leq \bigvee_{N=1}^\infty \sum_{n=1}^N\npm(\mss_n).
\]
\end{proof}

We continue with a first rudimentary form of the monotone convergence theorem; see \cref{1_res:monotone_convergence_theorem} for the latter.

\begin{proposition}\label{1_res:monotone_increasing_theorem_for_m}
Let $\msm$ be a measure space. If $\seq{\mss}$ is a increasing sequence in $\alg$ such that $\bigcup_{n=1}^\infty \mss_n\in\alg$, then $\npm(\mss_n)\uparrow\npm\left(\bigcup_{n=1}^\infty \mss_n\right)$ in $\osext$.
\end{proposition}

Note that, since $\npm(\mss_n)\uparrow$ in $\osext$, $\bigvee_{n=1}^\infty\npm(\mss_n)$ does indeed exists in $\osext$ by part~\ref{1_part:completeness_properties_of_extended_space_1} of \cref{1_res:completeness_properties_of_extended_space}.

\begin{proof}
Let us first suppose that $\bigvee_{n=1}^\infty\npm(\mss_n)\in\os$. In this case, $\npm(\mss_n)\in\os$ for all $n\geq 1$. Set $\widetilde\mss_1\coloneqq\mss_1$ and $\widetilde\mss_n\coloneqq\mss_n\setminus \mss_{n-1}$ for $n\geq 2$.  Then $\seq{\widetilde\mss}$ is a pairwise disjoint sequence in $\alg$ such that $\bigcup_{n=1}^\infty\widetilde \mss_n=\bigcup_{n=1}^\infty \mss_n\in\alg$. Since $\npm(\widetilde\mss_n)=\npm(\mss_n)-\npm(\mss_{n-1})$ for $n\geq 2$ by part~\ref{1_part:all_to_be_expected_4} of \cref{1_res:all_to_be_expected}, we see that
\begin{align*}
\npm\left(\bigcup_{n=1}^\infty \mss_n\right)&=\npm\left(\bigcup_{n=1}^\infty \widetilde\mss_n\right)\\
&=\bigvee_{N=1}^\infty\sum_{n=1}^N\npm(\widetilde\mss_n)\\
&=\bigvee_{N=2}^\infty\sum_{n=1}^N\npm(\widetilde\mss_n)\\
&=\bigvee_{N=2}^\infty\left[ \npm(\mss_1) + \sum_{n=2}^N\left(\npm(\mss_n)-\npm(\mss_{n-1})\right) \right]\\
&=\bigvee_{N=2}^\infty\npm(\mss_n)\\
&=\bigvee_{N=1}^\infty\npm(\mss_n).
\end{align*}
If $\bigvee_{n=1}^\infty\npm(\mss_n)=\largest$, then certainly
\[
\bigvee_{n=1}^\infty\npm(\mss_n)\geq\npm\left(\bigcup_{n=1}^\infty \mss_n\right).
\]
On the other hand, by the monotonicity of the measure, we obviously have $\npm(\mss_n)\leq\npm(\bigcup_{n=1}^\infty \mss_n)$ for all $n\geq 1$. Hence also
\[
\bigvee_{n=1}^\infty\npm(\mss_n)\leq\npm(\bigcup_{n=1}^\infty \mss_n).
\]
\end{proof}

Naturally, \cref{1_res:monotone_increasing_theorem_for_m} implies a counterpart for decreasing sequences of measurable subsets. It is a rudimentary form of the dominated convergence theorem; see \cref{1_res:dominated_convergence_theorem} for the latter.

\begin{proposition}\label{1_res:monotone_decreasing_theorem_for_m}
Let $\msm$ be a measure space.

If $\seq{\mss}$ is a decreasing sequence in $\alg$ with $\npm(\mss_1)\in \os$ such that $\bigcap_{n=1}^\infty\mss_n\in\alg$, then $\npm(\mss_n)\downarrow\npm\left(\bigcap_{n=1}^\infty\mss_n\right)$ in $\os$.
\end{proposition}

\begin{proof}
Set $\mss\coloneqq \bigcap_{n=1}^\infty\mss_n$.
Since $\mss_1\setminus\mss_n\uparrow \bigcup_{n=1}^\infty (\mss_1\setminus\mss_n)=\mss_1\setminus \mss$, \cref{1_res:monotone_increasing_theorem_for_m}, combined with part~\ref{1_part:all_to_be_expected_4} of \cref{1_res:all_to_be_expected} and part~\ref{1_part:operations_in_extended_space_6} of \cref{1_res:operations_in_extended_space}, shows that
\begin{align*}
\npm(\mss_1)&=\npm(\mss)+\npm(\mss_1\!\setminus\!\mss)\\
&=\npm(\mss)+\bigvee_{n=1}^\infty \npm(\mss_1\!\setminus\!\mss_n)\\
&=\npm(\mss)+\bigvee_{n=1}^\infty \left(\npm(\mss_1)-\npm(\mss_n)\right)\\
&=\npm(\mss)+\npm(\mss_1)+\bigvee_{n=1}^\infty(-\npm(\mss_n))\\
&=\npm(\mss)+ \npm(\mss_1)-\bigwedge_{n=1}^\infty\npm(\mss_n)
\end{align*}
in $\osext$. Since $\npm(\mss_1)$ and $\bigwedge_{n=1}^\infty\npm(\mss_n)$ have an inverse in $\osext$, we can now conclude that $\bigwedge_{n=1}^\infty\npm(\mss_n)=\npm(\mss)$.
\end{proof}

The final result of this section is a generalisation of the Borel-Cantelli lemma. We refer to  \cite[Exercise~1.2.89]{bogachev_MEASURE_THEORY_VOLUME_I:2007} for the classical result for a probability measure. Note, however, that our measure need not be finite. This is the reason of the appearance of a finiteness condition in part~\ref{1_part:borel-cantelli_lemma_2} that is automatically satisfied for probability measures.

\begin{lemma}[Borel-Cantelli lemma]\label{1_res:borel-cantelli_lemma}
Let $\msm$ be a measure space, and let $\seq{\mss}$ be a sequence in $\alg$. Suppose that $\msstwo_k\coloneqq\bigcup_{n=k}^\infty \mss_n\in\alg$ for all $k\geq 1$, and that $\msstwo\coloneqq\bigcap_{k=1}^\infty \msstwo_k= \bigcap_{k=1}^\infty\bigcup_{n=k}^\infty \mss_n\in\alg$.
\begin{enumerate}
  \item\label{1_part:borel-cantelli_lemma_1}
  If $\bigvee_{N=1}^\infty\sum_{n=1}^N\npm(\mss_n)\in\os$, then $\npm(\msstwo)=0$.
  \item\label{1_part:borel-cantelli_lemma_2}
  If $\npm\left(\left(\bigcup_{n=1}^\infty \mss_n\right)\setminus \msstwo\right)\in\os$, then $\npm(\msstwo)\geq x$ in $\osext$ for all $ x\in \os$ with the property that $\npm(\mss_n)\geq x$ for all $n\geq 1$.
  \end{enumerate}
\end{lemma}

Note that in part~\ref{1_part:borel-cantelli_lemma_2} it is not asserted that $\npm(\msstwo)$ is finite.

\begin{proof}
We start by proving part~\ref{1_part:borel-cantelli_lemma_1}. It is clear from the fact that $\msstwo\subseteq \msstwo_k$ and \cref{1_res:all_to_be_expected} that
\begin{equation}\label{1_eq:borel-cantelli_1}
\npm(\msstwo)\leq\npm(\msstwo_k)\leq\bigvee_{N=k}^\infty\sum_{n=k}^N\npm(\mss_n)
\end{equation}
in $\osext$ for all $k\geq 1$.

Furthermore, for all $k\geq 2$, we have

\begin{align*}
\bigvee_{N=1}^\infty\sum_{n=1}^N\npm(\mss_n)&=\bigvee_{N=k}^\infty\sum_{n=1}^N\npm(\mss_n)\\
&=\bigvee_{N=k}^\infty\left(\sum_{n=1}^{k-1}\npm(\mss_n)+\sum_{n=k}^N\npm(\mss_n)\right)\\
&=\sum_{n=1}^{k-1}\npm(\mss_n)+\bigvee_{N=k}^\infty\sum_{n=k}^N\npm(\mss_n)
\end{align*}
in $\osext$. Since $\bigvee_{N=k}^\infty\sum_{n=k}^N\npm(\mss_n)$ is finite for all $k\geq 2$, combination with \cref{1_eq:borel-cantelli_1} yields that
\[
\sum_{n=1}^{k-1}\npm(\mss_n)\leq \bigvee_{N=1}^\infty\sum_{n=1}^N\npm(\mss_n)-\npm(\msstwo)
\]
in $\osext$ for all $k\geq 2$. Hence
\[
\bigvee_{N=1}^\infty\sum_{n=1}^N\npm(\mss_n)\leq \bigvee_{N=1}^\infty\sum_{n=1}^N\npm(\mss_n)-\npm(\msstwo)
\]
in $\osext$. Since $\bigvee_{N=1}^\infty\sum_{n=1}^N\npm(\mss_n)$ is finite, we see that $\npm(\msstwo)\leq 0$, and we conclude that $\npm(\msstwo)=0$.

We turn to part~\ref{1_part:borel-cantelli_lemma_2}. Since clearly $\npm(\msstwo_k)\geq \npm(\mss_k)\geq x$ for all $k\geq 1$, we have
\begin{equation}\label{1_eq:borel-cantelli_2}
\bigwedge_{k=1}^\infty \npm(\msstwo_k)\geq x,
\end{equation}
where we note that the infimum in the left hand side exists since $\npm(\msstwo_k)\downarrow$.

Since $\msstwo_k\setminus \msstwo\downarrow\emptyset$ and since $\npm(\msstwo_1\setminus \msstwo)=\npm\left(\bigcup_{n=1}^\infty \mss_n\setminus \msstwo\right)\in\os$ by assumption, \cref{1_res:monotone_decreasing_theorem_for_m} implies that
\begin{equation}\label{1_eq:borel-cantelli_3}
\bigwedge_{k=1}^\infty\npm(\msstwo_k\setminus \msstwo)=0.
\end{equation}

Combining \cref{1_eq:borel-cantelli_2,1_eq:borel-cantelli_3} with part~\ref{1_part:binary_operations_in_extended_space_2} of \cref{1_res:binary_operations_in_extended_space}, we see that
\[
\npm(\msstwo)=\npm(\msstwo)+\bigwedge_{k=1}^\infty\npm(\msstwo_k\setminus \msstwo)=\bigwedge_{k=1}^\infty\left(\npm(\msstwo)+\npm(\msstwo_k\setminus \msstwo)\right)=\bigwedge_{k=1}^\infty\npm(\msstwo_k)\geq x
\]
in $\osext$.
\end{proof}

\section{$\pososext$-valued outer measures}
\label{1_sec:outer_measures}

\noindent One of the Riesz representation theorems for positive operators in \cite{de_jeu_jiang:2021b}  is established using vector-valued outer measures. The present short section, which will not be used in the later sections of the present paper, contains the necessary preparations for this. It is a modest modification of
\cite[Section~14]{aliprantis_burkinshaw_PRINCIPLES_OF_REAL_ANALYSIS_THIRD_EDITION:1998}.

Throughout this section, $\pset$ is a set and $\os$ is a \smc\ partially ordered vector space.

We begin with our definition of an $\pososext$-valued outer measure.

\begin{definition}\label{1_def:outer_measure}
A map $\pom: 2^{\pset}\to\pososext$ is called an \emph{$\pososext$-valued outer measure} if
\begin{enumerate}
\item\label{1_part:outer_measure_1}
$\pom(\emptyset)=0$;
\item\label{1_part:outer_measure_2}
$\pom(\mss_1)\leq\pom(\mss_2)$ in $\osext$ for all $\mss_1,\mss_2\in 2^{\pset}$ such that $\mss_1\subseteq \mss_2$;
\item\label{1_part:outer_measure_3}
for every sequence $\seq{\mss}$ of subsets of $\pset$,
\begin{equation*}
\pom\left(\bigcup_{n=1}^\infty \mss_n\right)\leq\bigvee_{N=1}^\infty \sum_{n=1}^N\pom(\mss_n)
\end{equation*}
in $\osext$.
\end{enumerate}
\end{definition}

The combination of the parts~\ref{1_part:outer_measure_1} and~\ref{1_part:outer_measure_3} shows that $\pom$ is not only $\sigma$-sub-additive, but also finitely sub-additive.

As in the real case, a measure can be found from an outer measure. The proofs need only minor modifications.

\begin{definition}\label{1_def:outer_measurable_sets}
A subset $\mss$ of $\pset$ is called \emph{$\pom$-measurable} if, for all $\msstwo\subseteq \pset$,
\begin{equation}\label{1_eq:outer_measurable_sets}
\pom(\msstwo)=\pom(\msstwo\cap\mss)+\pom(\msstwo\cap\mss^\comp)
\end{equation}
in $\osext$.
\end{definition}

We let $\M$ denote the collection of all $\pom$-measurable subsets of $\pset$. Obviously, $\emptyset,\ \pset\in\M$. and equally obviously $\M$ is invariant under the taking of complements. We shall proceed to show that $\M$ is a $\sigma$-algebra, and that the restriction of $\pom$ to $\M$ is an $\pososext$-valued measure.

\begin{lemma}\label{1_res:almost_finite_additivity_of_outer_measure_on_M}
Let $\mss_1,\mss_2$ be subsets of $\pset$ such that $\mss_1\in \M$ and $\mss_1\cap\mss_2=\emptyset$. Then, for any subset $\msstwo$ of $\pset$,
\[
\pom\big(\msstwo\cap(\mss_1\cup\mss_2)\big)=\pom(\msstwo\cap\mss_1)+\pom(\msstwo\cap\mss_2)
\]
in $\osext$.
\end{lemma}

\begin{proof}
Using that $\mss_1$ is $\pom$-measurable and that $(\mss_1\cup\mss_2)\cap\mss_1^\comp=\mss_2$, we see that,  for any $\msstwo\subseteq\pset$
\begin{align*}
\pom\big(\msstwo\cap(\mss_1\cup\mss_2)\big)
&=\pom\big(\left[\msstwo\cap(\mss_1\cup\mss_2)\right]\cap\mss_1\big)+\pom\big(\left[\msstwo\cap(\mss_1\cup\mss_2)\right]\cap\mss_1^\comp \big)\\
&=\pom(\msstwo\cap\mss_1)+\pom(\msstwo\cap\mss_2).
\end{align*}
\end{proof}

Applying \cref{1_res:almost_finite_additivity_of_outer_measure_on_M} for $\msstwo=\pset$ yields the following.

\begin{corollary}\label{1_res:finite_additivity_of_outer_measure_on_M}
Let $\mss_1,\mss_2$ be subsets of $\pset$ such that $\mss_1\in \M$ and $\mss_1\cap\mss_2=\emptyset$.  Then $\pom(\mss_1\cup\mss_2)=\pom(\mss_1)+\pom(\mss_2)$ in $\osext$.
\end{corollary}

\begin{theorem}\label{1_res:the_outer_measurable_sets_with_the_outer measure_is_a_measure_space}
Let $\pset$ be a set, let $\os$ be a \smc\ partially ordered vector space, and let $\pom\colon2^{\pset}\to\pososext$ be an $\pososext$-valued outer measure.

Then the set $\M$ of $\pom$-measurable subsets of $\pset$ is a $\sigma$-algebra. Furthermore,  the restriction of $\pom$ to $\M$ is an $\pososext$-valued measure on $\M$.
\end{theorem}

\begin{proof}
We start by proving that $\M$ is a $\sigma$-algebra. We have already observed that $\emptyset\in\M$ and that $\M$ is invariant under the taking of complements, so it remains to be shown that $\M$ is invariant under the taking of countable unions.

We show first that $\M$ is invariant under the taking of finite unions. Let $\mss_1,\mss_2$ in $\M$. Set $\mss\coloneqq\mss_1\cup\mss_2$. Using the $\pom$-measurability of $\mss_1$ for the first and the fourth equality, and that of $\mss_2$ for the third equality, we have, for any $\msstwo\subseteq\pset$,
\begin{align*}
\pom(\msstwo\cap\mss)+\pom(\msstwo\cap\mss^\comp)
&=\pom\big(\left[\msstwo\cap\mss\right]\cap\mss_1\big)+\pom\big(\left[\msstwo\cap\mss\right]\cap\mss_1^\comp\big)\\
&\quad\quad+
\pom\big(\left[\msstwo\cap\mss^\comp\right]\cap\mss_1\big)+\pom\big(\left[\msstwo\cap\mss^\comp\right]\cap\mss_1^\comp\big)\\
&=\pom(\msstwo\cap\mss_1) + \pom\big(\left[\msstwo\cap\mss_1^\comp\right]\cap\mss_2\big)\\
&\quad\quad + \pom(\emptyset) + \pom\big(\left[\msstwo\cap\mss_1^\comp\right]\cap\mss_2^\comp\big)\\
&=\pom(\msstwo\cap\mss_1)+\pom(\msstwo\cap\mss_1^\comp)\\
&=\pom(\msstwo)
\end{align*}
in $\osext$. Hence $\mss_1\cup\mss_2\in\M$, as desired. It follows that $\M$ is closed under the taking of finite unions.

Now suppose that $\seq{\mss}$ is a sequence in $\M$. We are to show that $\bigcup_{n=1}^\infty\mss_n$ is $\pom$-measurable. Replacing $\mss_n$ with $\mss_n\setminus\bigcup_{k=1}^n\mss_k$ for $n\geq 2$, which we now already know to be $\pom$-measurable, we may and shall suppose that the $\mss_n$ are pairwise disjoint. Set $\mssthree\coloneqq\bigcup_{n=1}^\infty\mss_n$, and $\mssthree_N\coloneqq\bigcup_{n=1}^N\mss_n$ for $N\geq 1$; then the $\mssthree_N$ are  $\pom$-measurable.  Using the $\pom$-measurability of $\mssthree_N$ in the first step, the monotonicity of $\pom$ in the second step, and \cref{1_res:almost_finite_additivity_of_outer_measure_on_M} and the fact that $\M$ is closed under the taking of finite unions in the third step, we see that, for all $N\geq 1$ and $\msstwo\subseteq \pset$,
\begin{align*}
\pom(\msstwo)
&=\pom(\msstwo\cap \mssthree_N)+\pom(\mss\cap \mssthree_N^\comp)\\
&\geq\pom(\msstwo\cap \mssthree_N)+\pom(\mss\cap \mssthree^\comp)\\
&=\sum_{n=1}^N\pom(\msstwo\cap\mss_n)+\pom(\mss\cap \mssthree^\comp)
\end{align*}
in $\osext$. Since this is true for each $N\geq 1$, a combination with the $\sigma$-sub-additivity of $\pom$ implies that, for any $\msstwo\subseteq\pset$,
\begin{align*}
\pom(\msstwo)
&\geq\bigvee_{N=1}^\infty\sum_{n=1}^N\pom(\msstwo\cap\mss_n)+\pom(\msstwo\cap \mssthree^\comp)\\
&\geq\pom\left(\bigcup_{n=1}^\infty (\msstwo\cap\mss_n)\right)+\pom(\msstwo\cap \mssthree^\comp)\\
&=\pom(\msstwo\cap \mssthree)+\pom(\msstwo\cap \mssthree^\comp)\\
&\geq\pom(\msstwo)
\end{align*}
in $\osext$. Hence $\mssthree\in\M$, as desired.

We shall now show that $\pom$ is $\sigma$-additive on $\M$. Let $\seq{\mss}$ be a pairwise disjoint sequence in $\M$.  The monotonicity of $\pom$ and the finite additivity of $\pom$ on $\M$ that follows from  \cref{1_res:finite_additivity_of_outer_measure_on_M} show that, for all $N\geq 1$,
\begin{align*}
\pom\left(\bigcup_{n=1}^\infty\mss_n\right)&\geq\pom\left(\bigcup_{n=1}^N\mss_n\right)\\
&=\sum_{n=1}^N\pom(\mss_n)
\end{align*}
in $\osext$. Hence $\pom\left(\bigcup_{n=1}^\infty\mss_n\right)\geq\bigvee_{N=1}^\infty \sum_{n=1}^N\pom(\mss_n)$. Since the reverse inequality holds by the $\sigma$-sub-additivity of $\pom$, we see that $\pom$ is $\sigma$-additive on $\M$.
\end{proof}

The following basic property carries over as well from the real case. It implies that the restriction of $\pom$ to the $\pom$-measurable subsets of $\pset$ is a complete $\pososext$-valued measure.

\begin{lemma}\label{1_res:null_set_is_outer_measurable}
Let $\mss$ be a subset of $\pset$ such that $\pom(\mss)=0$. Then $\mss$ is $\pom$-measurable.
\end{lemma}

\begin{proof}
If $\pom(\mss)=0$, then, for any subset $\msstwo$ of $\pset$,
\[
\pom(\msstwo)\leq\pom(\msstwo\cap\mss)+\pom(\msstwo\cap\mss^\comp)\leq\pom(\mss)+\pom(\msstwo)=\pom(\msstwo).
\]
\end{proof}

\section{Integration with respect to an $\pososext$-valued measure}\label{1_sec:integration_with_respect_to_a_positive_os-valued_measure}

\noindent In this section, we define the order integral with respect to $\pososext$-valued measures. After that, we proceed to establish the three basic convergence theorems: the monotone convergence theorem (see \cref{1_res:monotone_convergence_theorem}), Fatou's lemma (see \cref{1_res:fatou_lemma}), and the dominated convergence theorem (see \cref{1_res:dominated_convergence_theorem}).
The analogues of the classical $\lebfont L^1$- and $\Ell^1$-spaces are introduced and some of their vector lattice properties are investigated; see \cref{1_res:basic_properties_of_integrablefun} and \cref{1_res:ellone_is_dedekind_complete}.

\subsection{Order integrals }\label{1_subsec:integrals}

Let $\msm$ be a measure space, where $\alg$ is now a $\sigma$-algebra and not merely an algebra, $\os$ is a \smc\ partially ordered vector space, and $\npm:\alg\to\pososext$ is an $\pososext$-valued measure. In this section, we shall introduce an integral on suitable real-valued functions that corresponds to these data.\footnote{The \cref{1_res:integrals_of_elementary_functions,1_res:key_lemma_for_well-definednes_of_the_integral} are actually still valid when $\alg$ is an algebra and $\npm$ is finitely additive.} For some of the convergence theorems involving non-negative functions, such as the monotone convergence theorem, it is, in fact, more convenient (and more natural) to also allow functions taking values in the extended positive real numbers. This will, therefore, be our starting point.

While introducing some notation at the same time, we now start with the usual definitions and el\-e\-men\-ta\-ry results, referring to, e.g., \cite[p.~49-52]{bauer_MEASURE_AND_INTEGRATION_THEORY:2001} for details.

We supply $\posRext$ with the topology of the one-point compactification of $\posR$, and we let $\posextmeasfun$ denote the set of $\alg$--Borel-measurable functions $f:\pset\to\posRext$. A function $f:\pset\to\posRext$ is  $\alg$--Borel-measurable if and only if $\{\pt\in\pset : f(\pt) < r\}\in\alg$ for all (finite) $r\in\posR$.
The set $\posextmeasfun$ contains the pointwise sum, product, supremum, and infimum in $\posRext$ of two of its elements, and it is invariant under the pointwise action of $\posRext$. Every at most countable subset $\{f_n\: n\geq 1\}$ of $\posextmeasfun$ has a supremum and an infimum in $\posextmeasfun$, which is given by its pointwise supremum resp.\ infimum in $\posRext$. Hence the notation $f_n\uparrow f$ can be used to express a pointwise property in $\posRext$ as well as a fact in the partially ordered set  $\posextmeasfun$.

An element $\varphi$ of $\posextmeasfun$ is an \emph{elementary function} if it takes only finitely many values, which are all finite. 
When $S$ is a subset of $\pset$, then we let $\indicator{S}$ denote its indicator function, so that $\varphi$ can (non-uniquely) be written  as a finite sum $\varphi=\sum_{i=1}^n r_i\indicator{\mss_i}$ for some $n\geq 1$, $r_1,\dotsc,r_n\in\posR$, and $\mss_1,\dots,\mss_n\in\alg$. Here the $r_i$ are all finite, but it is allowed that $\npm(\mss_i)=\infty$ for some of the $\mss_i$. We let $\elemfun$ denote the set of elementary functions. It contains the pointwise sum, product, supremum, and infimum in $\posR$ of two of its elements, and it is invariant under the pointwise action of $\posR$.

If $\varphi=\sum_{i=1}^n r_i\indicator{\mss_i}$ is an elementary function, where the $\mss_i$ have been chosen to be pairwise disjoint, then we define its \emph{order integral}, which is an element of $\pososext$, by
\[
\ointm{\varphi}\coloneqq\sum_{i=1}^n r_i\npm(\mss_i),
\]
where $r_i\npm(\mss_i)$ refers to the action of $\posR$ as monoid ho\-mo\-mor\-phisms on $\osext$. We have added a superscript to indicate that the integral that we shall introduce for more general functions is defined using order properties. In contexts where $\os$ is, in fact, a partially ordered Banach space, integrals with respect to vector measures can then also be defined by using norm convergence rather than the ordering; this notation keeps the distinction clear. We shall return to the connection between these two types of measures and their integrals in \cref{1_sec:comparison_with_banach_space_case}.

The facts that $\npm$ is finitely additive and that $\posR$ acts as monoid ho\-mo\-mor\-phisms on $\osext$ imply that the integral does not depend on the choice for the pairwise disjoint $\mss_i$. The proof of this is exactly as the proof of \cite[Lemma~10.2]{bauer_MEASURE_AND_INTEGRATION_THEORY:2001} for $\os=\RR$. These two facts, combined with the fact that $(rs)x=r(sx)$ for all $r,\,s\in\posR$ and $x\in\osext$, also yield the following result, where the equalities and the inequality are in $\osext$.

\begin{lemma}\label{1_res:integrals_of_elementary_functions}
Let $\msm$ be a measure space, where $\alg$ is a $\sigma$-algebra.
Let $\varphi,\,\psi\in\elemfun$, and let $r\geq 0$. Then:
\begin{enumerate}
\item\label{1_part:integrals_of_elementary_functions_1}
$\ointm{r\varphi}=r\ointm{\varphi}$;
\item\label{1_part:integrals_of_elementary_functions_2} $\ointm{(\varphi+\psi)}=\ointm{\varphi}+\ointm{\psi}$;
\item\label{1_part:integrals_of_elementary_functions_3}
if $\varphi\leq\psi$, then $\ointm{\varphi}\leq\ointm{\psi}$.
\end{enumerate}
\end{lemma}

It follows from this that, for $\varphi\in\elemfun$, $\ointm{\varphi}=\sum_{i=1}^n r_i\npm(\mss_i)$ whenever $\varphi=\sum_{i=1}^n r_i\indicator{\mss_i}$ for not necessarily disjoint $\mss_i\in\alg$.
The proofs for all this are exactly as in \cite[p.~55-56]{bauer_MEASURE_AND_INTEGRATION_THEORY:2001}.

We shall define the order integral of an arbitrary $f\in\posextmeasfun$ in the natural way. To know that it is well defined, we need the following preparatory result. The proof is similar to that of \cite[Theorem~11.1]{bauer_MEASURE_AND_INTEGRATION_THEORY:2001} for the real case, but a comparison will show that it is still not a mere translation. The proof below for the general case will, in fact, enable one to argue that it is the Archimedean property of the real numbers that underlies the well-definedness of the integral also in this case, and not the continuity of the multiplication in the real numbers, as might be a possible interpretation of the proof of \cite[Theorem~11.1]{bauer_MEASURE_AND_INTEGRATION_THEORY:2001}.

\begin{lemma}\label{1_res:key_lemma_for_well-definednes_of_the_integral} Let $\msm$ be a measure space, where $\alg$ is a $\sigma$-algebra. Let $\varphi\in\elemfun$ and let $\seq{\varphi}\subseteq\elemfun$ be such that $\varphi_n\uparrow$ and $\varphi\leq\sup_{n\geq 1}\varphi_n$ pointwise in $\posRext$. Then
\[
\ointm{\varphi}\leq\bigvee_{n=1}^\infty\ointm{\varphi_n}
\]
in $\osext$.
\end{lemma}

\begin{proof}
We can write $\varphi=\sum_{i=1}^mr_i \chi_{\mss_i}$ for some $r_i\in\posR$ and $\mss_i\in\alg$. Let $\varepsilon$ be fixed such that $0<\varepsilon<1$. For $n\geq 1$, set $\msstwo_n\coloneqq\{\pt\in\pset : \varphi_n(x)\geq (1-\varepsilon)\varphi(\pt)\}$. Then $\msstwo_n$ is measurable and $\msstwo_n\uparrow \pset$, since $\varphi\leq\sup_{n\geq 1}\varphi_n$ and $\varphi_n\uparrow$. This implies that $\msstwo_n\cap \mss_i\uparrow \mss_i$ for $i=1,\cdots,m$, and then \cref{1_res:monotone_increasing_theorem_for_m} shows that $\npm(\mss_i\cap \msstwo_n)\uparrow\npm(\mss_i)$ in $\osext$ for $i=1,\dotsc,m$.
Since $\varphi_n\geq(1-\varepsilon)\indicator{\msstwo_n}\varphi$, \cref{1_res:integrals_of_elementary_functions} shows that
\begin{equation}\label{1_eq:epsilon_equation}
\ointm{\varphi_n}\geq(1-\varepsilon)\ointm{\indicator{\msstwo_n}\varphi}
\end{equation}
for $n\geq 1$. Furthermore,
\[
\ointm{\indicator{\msstwo_n}\varphi}=\ointm{\sum_{i=1}^m r_i\indicator{\msstwo_n}\indicator{\mss_i}}=\ointm{\sum_{i=1}^m r_i\indicator{\msstwo_n\cap\mss_i}}=\sum_{i=1}^m{r_i\npm(\msstwo_n\cap\mss_i)}.
\]
Since $\msstwo_n\cap \mss_i\uparrow \mss_i$ for $i=1,\cdots,m$, part~\ref{1_part:operations_in_extended_space_7} of \cref{1_res:operations_in_extended_space} and part~\ref{1_part:binary_operations_in_extended_space_2} of \cref{1_res:binary_operations_in_extended_space} then yield that
\begin{align*}
\bigvee_{n=1}^\infty\ointm{\indicator{\msstwo_n}\varphi}&=\bigvee_{n=1}^\infty\sum_{i=1}^m r_i\npm(\msstwo_n\cap\mss_i) =\sum_{i=1}^m r_i\bigvee_{n=1}^\infty \npm(\msstwo_n\cap\mss_i)\\
&=\sum_{i=1}^m r_i\npm(\mss_i)=\ointm{\varphi}.
\end{align*}
Combining this with \cref{1_eq:epsilon_equation}, we see that
\[
\bigvee_{n=1}^\infty\ointm{\varphi_n}\geq(1-\varepsilon)\ointm{\varphi}.
\]
in $\osext$. If $\ointm{\varphi}=\largest$, then we take $\varepsilon=1/2$ to see that $\bigvee_{n=1}^\infty\ointm{\varphi_n}=\largest$; we then have equality in the lemma. If $\ointm{\varphi}$ is finite, then we can write
\[
\bigvee_{n=1}^\infty\ointm{\varphi_n}-\ointm{\varphi}\geq -\varepsilon\ointm{\varphi}.
\]
Since this is true for every $\varepsilon$ such that $0<\varepsilon<1$, this implies that
\[
\bigvee_{n=1}^\infty\ointm{\varphi_n}-\ointm{\varphi}\geq-\bigwedge_{k=2}^\infty\,\, \frac{1}{k}\ointm{\varphi}.
\]
Since $\os$ is Archimedean, the right hand side is zero. This concludes the proof.
\end{proof}

Suppose now that $\alg$ is a $\sigma$-algebra, and let $f\in\posextmeasfun$. There exists a sequence $\seq{\varphi}\subseteq\elemfun$ such that $\varphi_n\uparrow f$ pointwise in $\posRext$; see \cite[Proposition~2.1.7]{cohn_MEASURE_THEORY_BIRKHAUSER_REPRINT:1993}, for example.
We define the \emph{order integral of $f$}, which is an element of $\pososext$, by setting
\[
\ointm{f}\coloneqq \bigvee_{n=1}^\infty \ointm{\varphi_n}.
\]
Since $\ointm{\varphi_n}\uparrow$, the \smc ness of $\os$ guarantees that this supremum exists in $\osext$; see part~\ref{1_part:completeness_properties_of_extended_space_1} of \cref{1_res:completeness_properties_of_extended_space}. To show that this definition is independent of the choice for the sequence $\seq{\varphi}$, let $\seq{\psi}\subset\elemfun$ be a second sequence such that $\psi_n\uparrow f$. Then $\varphi_k\leq f=\sup_{n\geq 1}\psi_n$ pointwise in $\posRext$ for all $k\geq 1$, so that \cref{1_res:key_lemma_for_well-definednes_of_the_integral} shows that $\ointm{\varphi_k}\leq\bigvee_{n=1}^\infty \ointm{\psi_n}$ in $\osext$. Hence $\bigvee_{k=1}^\infty\ointm{\varphi_k}\leq\bigvee_{n=1}^\infty \ointm{\psi_n}$ in $\osext$. The reverse inequality is likewise true, and we conclude that $\ointm{f}$ is well defined as an element of $\pososext$.

The integral has the usual properties as in the next  result. We include the easy proofs for the sake of completeness. Given the \cref{1_res:operations_in_extended_space,1_res:binary_operations_in_extended_space}, the proof is analogous to that for the real case.

\begin{lemma}\label{1_res:integral_on_positive_functions_has_usual_properties}
Let $\msm$ be a measure space, where $\alg$ is a $\sigma$-algebra,
let $f_1,\,f_2\in\posextmeasfun$, and let $r_1,\,r_2\in\posR$. Then:
\begin{enumerate}
\item\label{1_part:integral_on_positive_functions_has_usual_properties_1} $\ointm{(r_1f_1+r_2f_2)}=r_1\ointm{f_1}+r_2\ointm{f_2}$ in $\osext$;
\item\label{1_part:integral_on_positive_functions_has_usual_properties_2}
If $f_1\leq f_2$ pointwise in $\posRext$, then $\ointm{f_1}\leq\ointm{f_2}$ in $\osext$.
\end{enumerate}
\end{lemma}

\begin{proof}
Choose a sequence $\seq{\varphi}\subseteq\elemfun$ such that $\varphi_n\uparrow f_1$ pointwise in $\posRext$, and a sequence $\seq{\psi}\subseteq\elemfun$ such that $\psi_n\uparrow f_2$ pointwise in $\posRext$. Then $r_1\varphi_n+r_2\psi_n\uparrow r_1f_1+r_2f_2$ pointwise in $\posRext$, so that part~\ref{1_part:integral_on_positive_functions_has_usual_properties_1} follows from the definition of the integral, combined with part~\ref{1_part:operations_in_extended_space_7} of \cref{1_res:operations_in_extended_space} and part~\ref{1_part:binary_operations_in_extended_space_2} of \cref{1_res:binary_operations_in_extended_space}.

Since, for all $n\geq 1$, $\varphi_n\leq f_2=\sup_{n\geq 1}{\psi_n}$ pointwise in $\posRext$, \cref{1_res:key_lemma_for_well-definednes_of_the_integral} yields that $\ointm{\varphi_n}\leq\bigvee_{n=1}^\infty\ointm{\psi_n}=\ointm{f_2}$. Hence $\ointm{f_1}=\bigvee_{n=1}^\infty\ointm{\varphi_n}\leq\ointm{f_2}$, which is part~\ref{1_part:integral_on_positive_functions_has_usual_properties_2}.
\end{proof}

The importance of the Archimedean property of $\os$\textemdash and then also that of the real numbers\textemdash is again illustrated in the proof of part~\ref{1_part:finite_integral_implies_almost_everywhere_finite_1} of the following result.

\begin{lemma}\label{1_res:finite_integral_implies_almost_everywhere_finite}
Let $\msm$ be a measure space, where $\alg$ is a $\sigma$-algebra, and let $f\in\posextmeasfun$ be such that $\ointm{f}$ is finite.
Then:
\begin{enumerate}
\item\label{1_part:finite_integral_implies_almost_everywhere_finite_1}
$f$ is almost everywhere finite-valued;
\item\label{1_part:finite_integral_implies_almost_everywhere_finite_2}
the subset $\{\pt\in\pset : f(\pt)>0\}$ is $\sigma$-finite.
\end{enumerate}
\end{lemma}

\begin{proof}
We prove part~\ref{1_part:finite_integral_implies_almost_everywhere_finite_1}. For $n\geq 1$, set $\mss_n\coloneqq\{\pt\in\pset : f(x)\geq n\text{ in }\posRext\}$. Then $n\indicator{\mss_n}\leq f$, so that $n\npm(\mss_n)=\ointm{n \indicator{\mss_n}}\leq \ointm{f}$ by part~\ref{1_part:integral_on_positive_functions_has_usual_properties_2} of \cref{1_res:integral_on_positive_functions_has_usual_properties}. Hence $\npm(\mss_n)\leq 1/n\ointm{f}$ for all $n\geq 1$. From $\{\pt\in\pset : f(x)=\infty\}=\bigcap_{n=1}^\infty\mss_n$, we see that $\npm\{\pt\in\pset : f(x)=\infty\}\leq 1/n \ointm{f}$ for all $n\geq 1$. Since $\ointm{f}$ is finite, the Archimedean property of $\os$ then implies that $\npm\{\pt\in\pset : f(x)=\infty\}\leq 0$, which shows that $\npm\{\pt\in\pset : f(x)=\infty\}=0$.

For part~\ref{1_part:finite_integral_implies_almost_everywhere_finite_2}, set $\msstwo_n\coloneqq\{\pt\in\pset : f(x)\geq 1/n \text{ in }\posRext\}$ for $n\geq 1$. Then one sees similarly that $\npm(\msstwo_n)\leq n\ointm{f}$, which is finite. Since $\{\pt\in\pset : f(\pt)>0\}=\bigcup_{n=1}^\infty \msstwo_n$, part~\ref{1_part:finite_integral_implies_almost_everywhere_finite_2} is clear.
\end{proof}

\begin{lemma}\label{1_res:zero_integral_and_almost_everywhere_zero}
Let $\msm$ be a measure space, where $\alg$ is a $\sigma$-algebra.
Let $f\in\posextmeasfun$. Then the following are equivalent:
\begin{enumerate}
\item\label{1_part:zero_integral_and_almost_everywhere_zero_1}
$\ointm{f}=0$;
\item\label{1_part:zero_integral_and_almost_everywhere_zero_2}
$f(\pt)=0$ for almost all $\pt\in\pset$.
\end{enumerate}
\end{lemma}

\begin{proof}
Choose a sequence $\seq{\varphi}\subset\elemfun$ such that $\varphi_n\uparrow f$ pointwise in $\posRext$.

 Suppose that part~\ref{1_part:zero_integral_and_almost_everywhere_zero_1} holds. Then $\ointm{\varphi_n}=0$ for all $n\geq 1$ by  part~\ref{1_part:integral_on_positive_functions_has_usual_properties_2} of \cref{1_res:integral_on_positive_functions_has_usual_properties}. For the elementary functions $\varphi_n$, however, it is a direct consequence of the definition of their integrals that then $\npm(\{\pt\in\pset : \varphi_n(x)\neq 0\})=0$. Since the set $\{\pt\in\pset : f(x)\neq 0\}$ can be written as the union $\bigcup_{n=1}^\infty\{\pt\in\pset : \varphi_n(x)\neq 0\}$, the $\sigma$-sub-additivity of $\npm$ then implies that $\npm(\{\pt\in\pset : f(x)\neq 0\}=0$. Hence part~\ref{1_part:zero_integral_and_almost_everywhere_zero_1} implies part~\ref{1_part:zero_integral_and_almost_everywhere_zero_2}.

Suppose that part~\ref{1_part:zero_integral_and_almost_everywhere_zero_2} holds.
Since $0\leq\varphi_n(\pt)\leq f(\pt)$ for all $\pt\in\pset$, we see that $\varphi_n(\pt)=0$ for almost all $\pt\in\pset$. For the elementary functions $\varphi_n$, however, it is a direct consequence of the definition of their integrals that then $\ointm{\varphi_n}=0$ for all $n$. Since $\ointm{\varphi_n}\uparrow \ointm{f}$, it follows that $\ointm{f}=0$. Hence part~\ref{1_part:zero_integral_and_almost_everywhere_zero_2} implies part~\ref{1_part:zero_integral_and_almost_everywhere_zero_1}.
\end{proof}

\begin{corollary}\label{1_res:almost_everywhere_equal_implies_equal_integral_positive_case}
Let $\msm$ be a measure space, where $\alg$ is a $\sigma$-algebra.
Let $f_1,\,f_2\in\posextmeasfun$ and suppose that $f_1(\pt)=f_2(\pt)$ for almost all $\pt\in\pset$. Then $\ointm{f_1}=\ointm{f_2}$ in $\osext$.
\end{corollary}

\begin{proof}
Set $\mss\coloneqq\{\pt\in\pset : f_1(\pt)\neq f_2(\pt)\}$. Then $\mss$ is a measurable subset of measure zero. There exists $g\in\posextmeasfun$ such that $f_1=g+f_1\indicator{\mss}$ and $f_2=g+f_2\indicator{\mss}$. Then \cref{1_res:integral_on_positive_functions_has_usual_properties} and \cref{1_res:zero_integral_and_almost_everywhere_zero} show that $\ointm{f_1}$ and $\ointm{f_2}$ are both equal to $\ointm{g}$.
\end{proof}

We shall now define the order integral on a space of finite-valued measurable functions that need not be positive.  We shall write  $\measfun$ for the vector lattice of all $\RR$-valued $\alg$--Borel measurable functions on $\pset$ and $\posmeasfun$ for its positive cone of all positive measurable functions.
A function $f:\pset\to\posR$ is measurable in the present sense precisely if it is measurable in the earlier sense as a map from $\pset$ into $\posRext$.

We let $\integrablefun$ denote the set of all $f\in\measfun$ such that $\ointm{\abs{f}}$ is finite, and write $\posintegrablefun$ for the set of all positive measurable $f$ with finite integral. It follows from \cref{1_res:integral_on_positive_functions_has_usual_properties} that $\integrablefun$ is an order ideal of $\posmeasfun$; its positive cone is $\posintegrablefun$.

For $f\in\integrablefun$, we choose $f_1,f_2\in\posintegrablefun$ such that $f=f_1-f_2$, and we define the \emph{order integral of $f$} as $\ointm{f}\coloneqq\ointm{f_1}-\ointm{f_2}$. It is an element of $\os$. Invoking \cref{1_res:integral_on_positive_functions_has_usual_properties}, the usual arguments show that this is well defined, and that it defines a positive operator from $\integrablefun$ into $\os$.

We single out the following result for reference purposes. This triangle inequality for the order integral in the case of vector lattices is easily verified by splitting a function into its positive and negative parts.

\begin{lemma}\label{1_res:triangle_inequality_for_order_integral_in_case_of_vector_lattices}
	Let $\msm$ be a measure space, where $\alg$ is a $\sigma$-algebra and $\os$ is a \sDc\ vector lattice. Then $\lrabs{\ointm{f}}\leq\ointm{\lrabs{f}}$ for $f\in\integrablefun$.
\end{lemma}

Returning to general \smc\ partially ordered vector spaces that need not be vector lattices, we set
\[
\aezerofun\coloneqq\{f\in\measfun : f(\pt)=0\text{ for almost all }\pt\in\pset\}.
\]

\cref{1_res:zero_integral_and_almost_everywhere_zero} and \cref{1_res:integral_on_positive_functions_has_usual_properties} imply that $\aezerofun\subseteq\integrablefun$ and that $\ointm{f}=0$ for $f\in\aezerofun$, and that $\ointm{f_1}\leq\ointm{f_2}$ when $f_1,\,f_2\in\integrablefun$ are such that $f_1\leq f_2$ almost everywhere.

 For \mc\ $\os$, we shall return to the vector lattice properties of $\aezerofun$ and $\integrablefun$ in \cref{1_subsec:L_1_spaces} after the monotone convergence theorem will have been established in \cref{1_subsec:convergence_theorems}. For the moment, we conclude this section with the following result. It involves the $\sigma$-order continuity of an operator from \cref{1_def:order_continuity}.

\begin{proposition}\label{1_res:push_forward}
Let $\msm$ be a measure space, where $\alg$ is a $\sigma$-algebra and $\npm$ is finite, let $\ostwo$ be a \smc\ partially ordered vector space, and let $T:\os\to \ostwo$ be a $\sigma$-order continuous positive operator. Set
\[
\npm_T(\mss)\coloneqq T\left(\npm(\mss)\right)
\]
for $\mss\in\alg$. Then $(\pset,\alg,\npm_T,\ostwo)$ is a measure space.

Suppose that $f\in\posextmeasfun$ is such that $\ointm{f}\in\pososext$ is actually finite. Then $\int_\pset^{\mathrm o}\!f\di{\npm_T}\in\overline{\pos{\ostwo}}$ is also finite, and
\begin{equation}\label{1_eq:push_forward}
T\left(\ointm{f}\right)=\int_\pset^{\mathrm o}\!f\di{\npm_T}.
\end{equation}

Suppose that $f\in\integrablefun$. Then $f\in{\lebfont L}^1(\pset,\alg,\npm_T;\RR)$ and \cref{1_eq:push_forward} holds.
\end{proposition}

\begin{proof}
It is immediate from the $\sigma$-order continuity of $T$ that $\npm_T$ is an $\overline{\pos{\ostwo}}$-valued measure.

The validity of \cref{1_eq:push_forward} is clear for elementary functions. The validity for general $f\in\posextmeasfun$ then follows from the definition of the order integral and the $\sigma$-order continuity of $T$. This, in turn, implies the statement for $f\in\integrablefun$. 	
	
\end{proof}

\subsection{Convergence theorems}\label{1_subsec:convergence_theorems}

We shall now establish the three basic convergence theorems in the general context and start with the monotone convergence theorem.

\begin{theorem}[Monotone convergence theorem]\label{1_res:monotone_convergence_theorem}
Let $\msm$ be a measure space, where $\alg$ is a $\sigma$-algebra.
Let $\seq{f}\subseteq \posextmeasfun$ and $f\in\posextmeasfun$ be such that $f_n(\pt)\uparrow f(\pt)$ in $\posRext$ for almost all $\pt\in\pset$. Then
\[
\ointm{f_n}\uparrow\ointm{f}
\]
in $\osext$.
\end{theorem}

\begin{proof}
In view of \cref{1_res:zero_integral_and_almost_everywhere_zero}, we can, by redefining all $f_n$ and $f$ to be zero on a measurable subset of measure 0, suppose that $f_n(\pt)\uparrow f(\pt)$ in $\posRext$ for all $\pt\in\pset$.
For each $n\geq 1$, let $\{\varphi_i^n\}_{i=1}^\infty$ be a sequence in $\elemfun$ such that $\varphi_i^n\uparrow f_n$ pointwise in $\posRext$ as $i\to\infty$. Set $\psi_n\coloneqq\bigvee_{i=1}^n \varphi_n^i$ for $n\geq 1$.  Then $\psi_n\in\elemfun$ for all $n\geq 1$ and $\psi_n\uparrow f$ pointwise in $\osext$. Hence $\ointm{f}=\bigvee_{n=1}^\infty\ointm{\psi_n}$ by definition. On the other hand, since, for all $n\geq 1$, $\psi_n\leq f_n$ pointwise in $\osext$, we have $\ointm{\psi_n}\leq\ointm{f_n}$ for all $n\geq 1$. Hence $\ointm{f}\leq \bigvee_{n=1}^\infty\ointm{f_n}$ in $\osext$. As the reverse inequality is clear, the proof is complete.
\end{proof}

Just as \cref{1_res:monotone_increasing_theorem_for_m} implies \cref{1_res:monotone_decreasing_theorem_for_m}, \cref{1_res:monotone_convergence_theorem} implies our next result. It is a special case of the dominated convergence theorem; see part~\ref{1_part:dominated_convergence_theorem_4} of \cref{1_res:dominated_convergence_theorem}.

\begin{corollary}\label{1_res:monotone_convergence_theorem_for_decreasing_sequences}
Let $\msm$ be a measure space, where $\alg$ is a $\sigma$-algebra.
Let $\seq{f}\subseteq \posextmeasfun$ and $f\in\posextmeasfun$ be such that $f_n(\pt)\downarrow f(\pt)$ in $\posRext$ for almost all $\pt\in\pset$. If $\ointm{f_1}$ is finite, then
\[
\ointm{f_n}\downarrow\ointm{f}
\]
in $\os$.
\end{corollary}

\begin{proof}
In view of \cref{1_res:finite_integral_implies_almost_everywhere_finite} and \cref{1_res:almost_everywhere_equal_implies_equal_integral_positive_case}, we can, after redefining all $f_n$ and $f$ to be zero on a suitable measurable subset of measure zero, suppose that the $f_n$ have finite values and that $f_n\downarrow f$ pointwise. Then the functions $f_1-f_n$ are well-defined elements of $\posmeasfun$. Since $\ointm{f_n}$ is also finite for all $n$, we have $\ointm{(f_1-f_n)}=\ointm{f_1}-\ointm{f_n}$ for all $n\geq 1$. Similarly, $\ointm{(f_1-f)}=\ointm{f_1}-\ointm{f}$. Since $(f_1-f_n)\uparrow (f_1-f)$, an application of  \cref{1_res:monotone_convergence_theorem} shows that $\left(\ointm{f_1}-\ointm{f_n}\right)\uparrow(\ointm{f_1}-\ointm{f})$ in $\os$. We conclude that $\ointm{f_n}\downarrow\ointm{f}$.
\end{proof}

The combination of \cref{1_res:finite_integral_implies_almost_everywhere_finite} and \cref{1_res:monotone_convergence_theorem} yields the following.

\begin{proposition}\label{1_res:almost_everywhere_finite_supremum}
Let $\msm$ be a measure space, where $\alg$ is a $\sigma$-algebra.
Let $\seq{f}\subseteq \posextmeasfun$ and $f\in\posextmeasfun$ be such that $f_n(\pt)\uparrow f(\pt)$ in $\posRext$ for almost all $\pt\in\pset$, and suppose that $\bigvee_{n=1}^\infty\ointm{f_n}$ is finite. Then $\{\pt\in\pset: \bigvee_{n=1}^\infty f_n(\pt)=\largest \}$ is a measurable subset of $\pset$ of measure zero.
\end{proposition}

We continue with Fatou's lemma.

\begin{theorem}[Fatou's lemma]\label{1_res:fatou_lemma}
Let $\msm$ be a measure space, where $\alg$ is a $\sigma$-algebra and $\os$ is \sDc.

If $\seq{f}$ is a sequence in $\posextmeasfun$, then
\[
\ointm{\liminf_{n\to\infty}f_n}\leq\bigvee_{n=1}^\infty\bigwedge_{k=n}^\infty\ointm{f_k}
\]
in $\osext$.
\end{theorem}

The \sDc ness is necessary to guarantee that $\bigwedge_{k=n}^\infty\ointm{f_k}$ exists for all $n\geq 1$; \smc ness is no longer sufficient here.

\begin{proof}
For $n\geq 1$, set $g_n\coloneqq \inf_{k\geq n}f_k$. Then $0\leq g_n\leq f_k$ all $k\geq n$, so that \cref{1_res:integral_on_positive_functions_has_usual_properties} implies that
\begin{equation}\label{1_eq:fatou}
\ointm{g_n}\leq\bigwedge_{k\geq n}\ointm{f_k}
\end{equation}
in $\osext$ for all $n\geq 1$.

Applying \cref{1_res:monotone_convergence_theorem} to the increasing sequence $g_n\uparrow\liminf_{n\geq 1}f_n$, and using \cref{1_eq:fatou}, we then see that
\begin{align*}
\ointm{\liminf_{n\to\infty}f_n}&=\bigvee_{n=1}^\infty\ointm{g_n}\\
&\leq \bigvee_{n=1}^\infty\bigwedge_{k=n}^\infty\ointm{f_k}
\end{align*}
in $\osext$.
\end{proof}

We conclude with the dominated convergence theorem.

\begin{theorem}[Dominated convergence theorem]\label{1_res:dominated_convergence_theorem}
Let $\msm$ be a measure space, where $\alg$ is a $\sigma$-algebra and $\os$ is \sDc.

Let $\seq{f}$ be a sequence in $\measfun$, and let $f\in\measfun$ be such that $f_n(\pt)\to f(\pt)$ for almost all $\pt$ in $\pset$.

If there exists $g\in\posextmeasfun$ such that $\ointm{g}$ is finite, and such that, for all $n\geq 1$, $\abs{f_n(\pt)}\leq g(\pt)$ in $\posRext$ for almost all $\pt$ in $\pset$, then:
\begin{enumerate}
\item\label{1_part:dominated_convergence_theorem_1}
$f_n\in\integrablefun$ for all $n\geq 1$;
\item\label{1_part:dominated_convergence_theorem_2}
$f\in\integrablefun$;
\item\label{1_part:dominated_convergence_theorem_3} $\bigwedge_{n=1}^\infty\bigvee_{k=n}^\infty\ointm{\abs{f_k-f}}=0$;
\item\label{1_part:dominated_convergence_theorem_4}
\[
\ointm{f}=\bigvee_{n=1}^\infty\bigwedge_{k=n}^\infty\ointm{f_k}=\bigwedge_{n=1}^\infty\bigvee_{k=n}^\infty\ointm{f_k}.
\]
\end{enumerate}
\end{theorem}

\begin{proof}

In view of \cref{1_res:zero_integral_and_almost_everywhere_zero}, we may, by redefining all $f_n,\,f$, and $g$ to be zero on a measurable subset of measure zero, suppose that $g\in\posmeasfun$, that $\abs{f_n(\pt)}\leq g(\pt)$ for all $\pt\in\pset$, and that $f_n(\pt)\to f(\pt)$ for all $\pt\in\pset$. Since then also $\abs{f}(\pt)\leq g(\pt)$ for all $\pt\in\pset$, the $f_n$ and $f$ are in $\integrablefun$.

We turn to part~\ref{1_part:dominated_convergence_theorem_3}. Since $2g-\abs{f_n-f}\geq 0$ pointwise, \cref{1_res:fatou_lemma} shows that
\begin{align*}
\ointm{2g}&=\ointm{\liminf_{n\geq 1}\left(2g-\abs{f_n-f}\right)}\\
&\leq\bigvee_{n=1}^\infty\bigwedge_{k=n}^\infty\ointm{\left(2g-\abs{f_n-f}\right)}\\
&=\ointm{2g}-\bigwedge_{n=1}^\infty\bigvee_{k=n}^\infty\ointm{\abs{f_n-f}},
\end{align*}
where the final equality is valid since the integrals $\ointm{\left(g-\abs{f_n-f}\right)}$, $\ointm{g}$, and $\ointm{\abs{f_n-f}}$ all lie in the finite order interval $[0,2\ointm{g}]$ of $\os$. Cancelling the finite element $\ointm{2g}$, we see that $\bigwedge_{n=1}^\infty\bigvee_{k=n}^\infty\ointm{\abs{f_n-f}}\leq 0$. Since the reverse inequality is obvious, the proof of part~\ref{1_part:dominated_convergence_theorem_3} is complete.

We turn to part~\ref{1_part:dominated_convergence_theorem_4}.

Since $g+f_n\geq 0$ for all $n\geq 1$, Fatou's lemma shows that
\begin{align*}
\ointm{\left(g+f\right)}&=\ointm{\liminf_{n\geq 1}\left(g+f_n\right)}\\
&\leq\bigvee_{n=1}^\infty\bigwedge_{k=n}^\infty\ointm{\left(g+f_n\right)}\\
&=\ointm{g}+ \bigvee_{n=1}^\infty\bigwedge_{k=n}^\infty\ointm{f_n},
\end{align*}
from which we see that
\begin{equation}\label{1_eq:dominated_convergence_1}
\ointm{f}\leq \bigvee_{n=1}^\infty\bigwedge_{k=n}^\infty\ointm{f_n}.
\end{equation}

Since $g-f_n\geq 0$ for all $n\geq 1$, Fatou's lemma shows that
\begin{align*}
\ointm{\left(g-f\right)}&=\ointm{\liminf_{n\geq 1}\left(g-f_n\right)}\\
&\leq\bigvee_{n=1}^\infty\bigwedge_{k=n}^\infty\ointm{\left(g-f_n\right)}\\
&=\ointm{g}- \bigwedge_{n=1}^\infty\bigvee_{k=n}^\infty\ointm{f_n},
\end{align*}
from which we see that
\begin{equation}\label{1_eq:dominated_convergence_2}
\bigwedge_{n=1}^\infty\bigvee_{k=n}^\infty\ointm{f_n}\leq\ointm{f}.
\end{equation}

Combining \cref{1_eq:dominated_convergence_1,1_eq:dominated_convergence_2} with \cref{1_res:liminf_and_limsup}, we have
\[
\bigwedge_{n=1}^\infty\bigvee_{k=n}^\infty\ointm{f_n}\leq\ointm{f}\leq\bigvee_{n=1}^\infty\bigwedge_{k=n}^\infty\ointm{f_n}\leq\bigwedge_{n=1}^\infty\bigvee_{k=n}^\infty\ointm{f_n},
\]
which completes the proof of part~\ref{1_part:dominated_convergence_theorem_4}.
\end{proof}

\subsection{$\boldsymbol{\lebfont{L}^1}$-spaces and $\boldsymbol{\Ell^1}$-spaces}\label{1_subsec:L_1_spaces}

In this section, we collect some vector lattice properties of $\integrablefun$ and its quotient space $\ellone$ that will be defined below.

\begin{proposition}\label{1_res:basic_properties_of_integrablefun}
	Let $\msm$ be a measure space, where $\alg$ is a $\sigma$-algebra.
	\begin{enumerate}
		\item\label{1_part:basic_properties_of_integrablefun_1}
		$\integrablefun$ is an order ideal of the vector lattice $\measfun$. As a consequence,  it is a \sDc\ vector lattice.
		\item\label{1_part:basic_properties_of_integrablefun_1_part:basic_properties_of_integrablefun_2}
		The positive operator $f\mapsto\ointm{f}$ from $\integrablefun$ into $\os$ is $\sigma$-order continuous.
		\item\label{1_part:basic_properties_of_integrablefun_1_part:basic_properties_of_integrablefun_3}
		$\aezerofun$ is a $\sigma$-order ideal of $\integrablefun$.
	\end{enumerate}
\end{proposition}

\begin{proof}
	It was already observed in \cref{1_subsec:convergence_theorems} that $\integrablefun$ is an order ideal of $\measfun$, and then it inherits the \sDc ness of $\measfun$.
	
It follows from	\cref{1_res:monotone_convergence_theorem_for_decreasing_sequences} that the order integral is $\sigma$-order continuous.

We have $\aezerofun=\{f\in\integrablefun : \ointm{\abs{f}}=0\}$ by \cref{1_res:zero_integral_and_almost_everywhere_zero}. Hence $\aezerofun$ is the null ideal of the order integral on $\integrablefun$. Since this is a $\sigma$-order continuous operator, it now follows that $\aezerofun$ is a $\sigma$-ideal of $\integrablefun$.
\end{proof}

We shall now introduce the generalisation of the classical $\Ell^1$-space to the vector-valued case.

Since $\aezerofun$ is an order ideal of $\integrablefun$, the quotient space
\[
\ellone\coloneqq\integrablefun/\aezerofun
\]
is again a vector lattice when it is supplied with the partial ordering that is defined by the image of the positive cone $\posintegrablefun$ of $\integrablefun$ under the quotient map. The quotient map is then a vector lattice homomorphism. We shall write $[f]$ for the image of $f\in\integrablefun$ under the quotient map. In \cref{1_res:ellone_is_dedekind_complete}, we shall give sufficient conditions on $\os$ for $\ellone$ to be \Dc. For this, we need preparations that are of some independent interest.

We say that partially ordered vector space has the \emph{countable sup property} when, for every net $\net{x}\subseteq\posos$ and $x\in\pos{\os}$ such that $x_\lambda\uparrow x$, there exists an at most countably infinite set of indices $\{\lambda_n: n\geq 1 \}$ such that $x=\sup_{n\geq 1} x_{\lambda_n}$. In this case, there also always exist indices $\lambda_1\leq\lambda_2\leq\dotsb$ such that $x_{\lambda_n}\uparrow x$.\footnote{Strictly speaking, it would be better to call this property the \emph{monotone} countable sup property. We have refrained from doing so to keep the terminology somewhat simpler.} For vector lattices, our countable sup property is equivalent to what is usually called the countable sup property in that context; namely, that every subset that has a supremum contains an at most countably infinite subset with the same supremum.  In some sources, a vector lattice with this property is then said to be order separable. As usual, we shall say that a positive operator between two partially ordered vector spaces is \emph{strictly positive} when the intersection of its kernel with the positive cone of the domain is $\{0\}$. The proof of our next result is inspired by \cite[p.~65--66]{zaanen_INTRODUCTION_TO_OPERATOR_THEORY_IN_RIESZ_SPACES:1997}.

\begin{lemma}\label{1_res:key_lemma_for_completeness}
	Let $\os$ be a \sDc\ vector lattice, let $\ostwo$ be a partially ordered vector space, and let $\posmapT:\os\to \ostwo$ be a strictly positive $\sigma$-order continuous operator. Suppose that $S$ is a non-empty subset of $\os$ that is bounded above in $\os$, and that $\seq{t}$ is a sequence in $S$ such that
	\begin{enumerate}
		\item\label{1_part:key_lemma_for_completeness_1}
		$S$ is closed under the taking of finite suprema;
		\item\label{1_part:key_lemma_for_completeness_2}
		$\sup \posmapT(S)$ and $\sup\{\posmapT(t_n): n\geq  1\}$ both exist in $\ostwo$ and are equal.
	\end{enumerate}
	Then $\sup S$ exists in $\os$. Moreover, if we set $s_1\coloneqq t_1$, $s_2\coloneqq t_1\vee t_2$, $s_3\coloneqq t_1\vee t_2\vee t_3$, \ldots, then $\seq{s}$ is a sequence in $S$ such that $s_n\uparrow \sup S$ in $\os$ and $\posmapT(s_n)\uparrow \sup \posmapT(S)$ in $\ostwo$. Consequently, $\posmapT(\sup S)=\sup\posmapT(S)$.
\end{lemma}

\begin{proof}
	Since $s_n\in S$ and $s_n\geq t_n$ for $n\geq 1$, it is clear that $\posmapT(s_n)\uparrow \sup\posmapT(S)$. Because $\seq{s}$ is an increasing sequence in the bounded above set $S$, there exists an $s_\infty\in\os$ such that $s_n\uparrow s_\infty$ in $\os$. Hence $\posmapT(s_n)\uparrow \posmapT(s_\infty)$, so that $\posmapT(s_\infty)=\sup\posmapT(S)$.
	
	We claim that $s_\infty$ is an upper bound for $S$. To see this, take $s\in S$. Since $s\vee s_n\uparrow s\vee s_\infty$, we have $\sup\{\posmapT(s\vee s_n) : n\geq 1\}=\posmapT(s\vee s_\infty)$. Using that $s\vee s_n\in S$ for $n\geq 1$, we have	
	\begin{align*}
		\sup\posmapT(S)&\geq \sup\{\posmapT(s\vee s_n) : n\geq 1\}\\
		&=\posmapT(s\vee s_\infty)\\
		&\geq\posmapT(s_\infty)\\
		&=\sup\posmapT(S).
	\end{align*}
	It follows that $\posmapT(s\vee s_\infty-s_\infty)=0$. Since $\posmapT$ is strictly positive, this implies that $s_\infty \geq s$, as desired.
	
	Let $u$ be an upper bound for $S$. Then certainly $u\geq s_n$ for $n\geq 1$, so that $u\geq s_\infty$.
	
	We conclude that $s_\infty=\sup S$. This completes the proof.
\end{proof}

The following is an easy consequence of \cref{1_res:key_lemma_for_completeness}. The special case where $\ostwo$ is a vector lattice follows from  \cite[Exercise~1.4.2.a]{aliprantis_burkinshaw_POSITIVE_OPERATORS_SPRINGER_REPRINT:2006}.

\begin{proposition}\label{1_res:completeness_pulled_back}
	Let $\os$ be a \sDc\ vector lattice, let $\ostwo$ be a partially ordered vector space that is \mc\ and has the countable sup property, and let $\posmapT:\os\to \ostwo$ be a strictly positive $\sigma$-order continuous operator.
	
	Then $\os$ is \Dc\ and has the countable sup property. Moreover, $\posmapT$ is order continuous.
\end{proposition}

\begin{proof}
Let $S$ be a non-empty subset of $\os$ that is bounded above. Set $S^\vee\coloneqq\{s_1\vee\dotsb s_k : k\geq1,\,s_1,\dotsc,s_k\in S\}$. Then $\sup S^\vee$ exists by \cref{1_res:key_lemma_for_completeness}. Hence $\os$ is \Dc. \cref{1_res:key_lemma_for_completeness} also supplies a sequence $\{s_n^\vee\}_{n=1}^\infty$ in $S^\vee$ such that $s_n^\vee\uparrow \sup S^\vee=\sup S$. This implies that $\os$ has the countable sup property.

To show that $\posmapT$ is order continuous, take a non-empty upward directed subset $S$ of $\os$ that is bounded above. Using the final statement in \cref{1_res:key_lemma_for_completeness}, we have  $\posmapT(\sup S)=\posmapT(\sup S^\vee)=\sup\posmapT(S^\vee)$. Because $S$ is upward directed, the subsets $S^\vee$ and $S$ of $\os$ are interlaced in the partial ordering on $\os$. Since $\posmapT$ is positive, the same is true for their images $\posmapT(S^\vee)$ and $\posmapT(S)$ in the partial ordering on $\ostwo$. Hence $\sup\posmapT(S^\vee)=\sup\posmapT(S)$. This completes the proof.
\end{proof}

We now come to the vector lattice properties of $\ellone$.

\begin{theorem}\label{1_res:ellone_is_dedekind_complete}
	Let $\msm$ be a measure space, where $\alg$ is a $\sigma$-algebra.
	
	The quotient map from $\integrablefun$ onto $\ellone$ is a $\sigma$-order continuous vector lattice homomorphism, and $\ellone$ is a \sDc\ vector lattice. The map $\opintm: \ellone\to\os$, defined by setting $\opintm([f])\coloneqq\ointm{f}$ for $[f]\in\ellone$, is well defined, linear, strictly positive, and $\sigma$-order continuous.
	
	If $\os$ is \mc\ and has the countable sup property, then $\ellone$ is a \Dc\  vector lattice with the countable sup property. Moreover, $\opintm$ is then order continuous.
\end{theorem}

\begin{proof}
		Since we know from part~\ref{1_part:basic_properties_of_integrablefun_1_part:basic_properties_of_integrablefun_3} of \cref{1_res:basic_properties_of_integrablefun} that $\aezerofun$ is a $\sigma$-order ideal of $\integrablefun$, it follows from
		\cite[Theorem~18.11]{luxemburg_zaanen_RIESZ_SPACES_VOLUME_I:1971} that the quotient map is $\sigma$-order continuous. Since an increasing sequence in $\ellone$ and an upper bound of it can be lifted to $\integrablefun$, the \sDc ness of $\integrablefun$ and the $\sigma$-order continuity of the quotient map then show that $\ellone$ is \sDc.
	
	It is clear that the operator $\opintm$ can be defined on $\ellone$ and that it is strictly positive. Its $\sigma$-order continuity is easily seen to follow from \cref{1_res:monotone_convergence_theorem_for_decreasing_sequences}.
	
	An appeal to \cref{1_res:completeness_pulled_back} yields the remainder of the statements.
\end{proof}

\section{Comparison with positive vector measures and their integrals}\label{1_sec:comparison_with_banach_space_case}

\noindent Let $\ms$ be a measurable space, where $\alg$ is a $\sigma$-algebra, and suppose that $\os$ is a \smc\ Banach space with a closed positive cone. In this section, we shall discuss the relation between the measures in the present paper and vector measures in the classical sense, as well as the relation between the corresponding integrals.

 We recall that a classical (as we shall call it for the sake of discussion) vector measure with values in $\os$ is a map $\npm:\alg\to\os$ such that $\npm(\emptyset)=0$ and
\begin{equation}\label{1_eq:vector_measure_sigma_additivity}
\npm\left(\bigcup_{n=1}^\infty\mss_n\right)=\sum_{n=1}^\infty\npm(\mss_n),	
\end{equation}
whenever $\seq{\mss}$ is a pairwise disjoint sequence in $\alg$. Here the series in the right hand side of \cref{1_eq:vector_measure_sigma_additivity} is convergent in the norm of $\os$. As is well known, the fact that complex measures are bounded (see \cite[Theorem~6.4]{rudin_REAL_AND_COMPLEX_ANALYSIS_THIRD_EDITION:1987}, for example), when combined with the uniform boundedness principle, implies that a vector measure is automatically norm bounded on $\alg$. This makes clear that the integral that is canonically defined on the elementary functions extends by continuity to a norm-to-norm continuous classical $\os$-valued integral on the bounded measurable functions on $\pset$.

 Obviously, there is no classical vector measure that can be an analogue of an infinite $\pososext$-valued measure in our sense. It is meaningful, however, to ask for the relation between $\posos$-valued measures in the sense of our \cref{1_def:positive_pososext_valued_measure} and classical $\posos$-valued vector measures.\footnote{Classical, not necessarily positive, vector measures with values in a Banach lattice are the subject of \cite[Chapter~III]{groenewegen_THESIS:1983} and \cite[Chapter~IV]{jeurnink_THESIS:1983}.} To clarify this, we start by observing that, for a net $\net{x}$ in $\os$ and an element $x$ of $\os$, the facts that $x_\lambda\uparrow$ and that $x_\lambda\to x$ in norm imply that also $x_\lambda\uparrow x$. This follows readily from the hypothesis that $\posos$ be closed. It is then immediate that a classical $\posos$-valued vector measure is also an $\posos$-valued measure in our sense. When the norm on $\os$ has the property that $\norm{x-x_n}\to 0$ for every sequence $\seq{x}$ in $\posos$ and $x\in\posos$ such that $x_n\uparrow x$, then the converse also holds, so that the two notions coincide. When the norm on the Banach space $\os$ fails to be $\sigma$-monotone order continuous in the sense as just described, then it can actually occur that there exists an $\posos$-valued measure in our sense that is not a classical $\posos$-valued vector measure. As an example, we take $\os=\ell^\infty$, and we let $\seq{e}$ denote the sequence of standard unit vectors in it. For $X$ we take $\NN$, and for $\alg$ we take the power set of $\NN$. For $\mss\in\alg$, we set $\npm(\mss)\coloneqq\bigvee_{n\in\mss}e_n$. Then $\mu$ is an $\pos{\left(\ell^\infty\right)}$-valued measure in our sense, but it is not a classical vector measure. Indeed, the terms of the series in the aspired equality $\npm(\NN)=\sum_{n=1}^\infty \npm(\{n\})=\sum_{n=1}^\infty e_n$ do not even converge to zero in norm.

 Because $\ell^\infty$ can be embedded isometrically as a Banach lattice into the regular operators on $\ell^p$ for $1\leq p\leq\infty$, this example also shows that the two types of measures do not even coincide when they are required to be operator-valued, which is one of our important classes of applications of the results of the current paper in \cite{de_jeu_jiang:2021b,de_jeu_jiang:2021c}. The reader may at this point wish to recall the \cref{1_res:measures_with_values_in_the_regular_operators,1_res:measures_with_values_in_L_sa}, pointing out the role of the strong operator topology rather than the uniform topology for our operator-valued measures.

 To continue our discussion, we take $\os=\ell^\infty$ and $X=\NN$ again, and we let $\alg$ be the power set of $\NN$ again. We define the map $\npm:\alg\to\posos$ by setting $\npm(\mss)\coloneqq\sum_{n\in\mss} e_n/n$ for $\mss\in\alg$. Then $\npm$ is an $\posos$-valued measure in our sense, as well as a classical $\posos$-valued vector measure. It is easy to see\textemdash by using the automatic continuity of positive operators between Banach lattices, for example\textemdash that the classical $\ell^\infty$-valued integral on the Banach lattice of all bounded functions on $\NN$ coincides with our order integral on that space. The space of positive functions that are order integrable is, however, larger than the bounded functions on $\NN$. The unbounded function $f_0:\NN\to\posR$, defined by setting $f_0(n)\coloneqq n$ for $n\in\NN$, is order integrable. Indeed, the elementary positive functions $\varphi_n:\NN\to\posR$ for $n\geq 1$, defined by setting $\varphi_n(k)\coloneqq k$ when $1\leq k\leq n$ and $\varphi_n(k)\coloneqq 0$ when $k>n$, increase pointwise to $f_0$ and the sequence $\left\{\sum_{k=1}^n e_k\right\}_{n=1}^\infty$ of their order integrals has  $\bigvee_{n=1}^\infty e_n$ as its supremum. Of course, there are natural Banach space methods to attempt to extend the domain of the classical integral with respect to the classical vector measure $\npm$. A possible approach is to consider those functions $f:\NN\to\posR$ for which there is a sequence $\seq{\varphi}$ of bounded functions on $\NN$ such that $\varphi_n\to f$ pointwise and such that the sequence $\left\{\int_\NN\!\varphi_n\di{\npm}\right\}_{n=1}^\infty$ of classical integrals is a Cauchy sequence in $\ell^\infty$. One would then have to verify that the limit of this sequence is independent of the choice of the $\varphi_n$, or perhaps restrict oneself to those $f$ for which is the case. Variations on this are also possible. One can, for example, require that the $\varphi_n$ be elementary functions, or start with positive functions and approximate these from below with positive elementary functions. Such Banach space approaches will, however, never lead to the definition of an integral for $f_0$. Indeed, the sequence $\{\varphi_n\}_{n=1}^\infty$ above consists of positive elementary functions and approximates $f_0$ from below, which is arguably the best approximation one could wish for. The sequence $\{\sum_{k=1}^n e_k\}_{n=1}^\infty$ of their classical integrals is, however, not a Cauchy sequence. We thus see that, even when a set map is an $\posos$-valued measure in our sense as well as a classical $\posos$-valued vector measure, it can still happen that the associated order integral is intrinsically more comprehensive than the associated classical integral.

The example in the preceding paragraph can only exist because the norm on $\ell^\infty$ is not $\sigma$-monotone order continuous. Indeed, suppose that the \smc\ partially ordered Banach space $\os$ with a closed positive cone has a $\sigma$-monotone order continuous norm, and that $\npm:\alg\to\posos$ is a set map. The fact that, for a sequence $\seq{x}$ in $\posos$ and an element $x$ of $\posos$, $x_n\uparrow x$ if and only if $x_n\to x$ in norm, does not only imply that $\npm$ is a classical vector measure if and only if it is an $\posos$-valued measure in our sense. It also guarantees that the extension procedure of the domain of the classical integral, using a pointwise approximation of a positive function from below by elementary functions, is, indeed, possible and yields the same set of integrable functions on which the two integral also agree. Since the norms on partially ordered Banach spaces of \emph{operators} are typically not $\sigma$-monotone order continuous, this observation does usually not apply there.

\section*{Acknowledgements}
\noindent The authors thank Marten Wortel for helpful discussions on JBW-algebras.

\bibliographystyle{plain}
\urlstyle{same}

\bibliography{general_bibliography}

\end{document}